\documentclass{article}

\usepackage{color}
\definecolor{black}{rgb}{0,0,0}\color{black}

\usepackage{times}
\usepackage{graphicx} 
\usepackage{subfigure} 
\usepackage{psfrag}

\usepackage{natbib}

\usepackage{algorithm}
\usepackage{algorithmic}
\usepackage{amsmath,amssymb,amsthm}
\algsetup{indent=1.5em}

\usepackage{boxedminipage}

\usepackage{hyperref}
\usepackage{balance}

\usepackage[accepted]{icml2017}

\icmltitlerunning{Exploiting Strong Convexity from Data with Primal-Dual First-Order Algorithms}

\newcommand{\argmin}{\mathop{\mathrm{arg\,min}{}}}
\newcommand{\argmax}{\mathop{\mathrm{arg\,max}{}}}
\newcommand{\prox}{\mathrm{prox}}

\newcommand{\eqdef}{\stackrel{\mathrm{def}}{=}}
\newcommand{\reals}{\mathbb{R}}
\newcommand{\R}{\mathbb{R}}

\newcommand{\E}{\mathbb{E}}

\newcommand{\lambdamin}{\lambda_\mathrm{min}}
\newcommand{\lambdamax}{\lambda_\mathrm{max}}

\newcommand{\Romannum}[1]{\uppercase\expandafter{\romannumeral#1}}

\newcommand{\Dini}{{\Delta^{(0)}}}
\newcommand{\Dt}{{\Delta^{(t)}}}
\newcommand{\Dtp}{{\Delta^{(t+1)}}}
\newcommand{\xini}{{x^{(0)}}}
\newcommand{\xopt}{{x^\star}}
\newcommand{\xhat}{\hat{x}}
\newcommand{\xtld}{\tilde{x}}
\newcommand{\xbar}{\bar{x}}
\newcommand{\xt}{{x^{(t)}}}
\newcommand{\xtp}{{x^{(t+1)}}}
\newcommand{\xtm}{{x^{(t-1)}}}
\newcommand{\xtldini}{\tilde{x}^{(0)}}
\newcommand{\xtldt}{\tilde{x}^{(t)}}
\newcommand{\xtldtp}{\tilde{x}^{(t+1)}}
\newcommand{\yini}{{y^{(0)}}}
\newcommand{\yopt}{{y^\star}}
\newcommand{\yhat}{\hat{y}}
\newcommand{\ytld}{\tilde{y}}
\newcommand{\ybar}{\bar{y}}
\newcommand{\yt}{{y^{(t)}}}
\newcommand{\ytp}{{y^{(t+1)}}}

\newcommand{\uopt}{{u^\star}}
\newcommand{\ut}{{u^{(t)}}}
\newcommand{\utp}{{u^{(t+1)}}}

\newcommand{\vini}{{v^{(0)}}}
\newcommand{\vt}{{v^{(t)}}}
\newcommand{\vtp}{{v^{(t+1)}}}

\newcommand{\Lagr}{\mathcal{L}}
\newcommand{\Bdiv}{\mathcal{D}}
\newcommand{\Bdivi}{\mathcal{D}_i}
\newcommand{\yiini}{{y_i^{(0)}}}
\newcommand{\yiopt}{{y_i^\star}}

\newcommand{\yitld}{\tilde{y}_i}

\newcommand{\yit}{{y_i^{(t)}}}
\newcommand{\yitp}{{y_i^{(t+1)}}}

\newcommand{\ykt}{{y_k^{(t)}}}
\newcommand{\yktp}{{y_k^{(t+1)}}}

\newcommand{\viini}{{v_i^{(0)}}}
\newcommand{\vit}{{v_i^{(t)}}}
\newcommand{\vitp}{{v_i^{(t+1)}}}
\newcommand{\vkt}{{v_k^{(t)}}}
\newcommand{\vktp}{{v_k^{(t+1)}}}
\newtheorem{lemma}{Lemma}
\newtheorem{theorem}{Theorem}
\newtheorem{assump}{Assumption}

\begin{document} 

\twocolumn[
\icmltitle{Exploiting Strong Convexity from Data with Primal-Dual\\ 
           First-Order Algorithms}




\begin{icmlauthorlist}
\icmlauthor{Jialei Wang}{uch}
\icmlauthor{Lin Xiao}{msr}
\end{icmlauthorlist}

\icmlaffiliation{uch}{Department of Computer Science, 
                      The University of Chicago, 
                      Chicago, Illinois 60637, USA.} 
\icmlaffiliation{msr}{Microsoft Research, Redmond, Washington 98052, USA}

\icmlcorrespondingauthor{Jialei Wang}{jialei@uchicago.edu}
\icmlcorrespondingauthor{Lin Xiao}{lin.xiao@microsoft.com}

\icmlkeywords{boring formatting information, machine learning, ICML}

\vskip 0.3in
]



\printAffiliationsAndNotice{}  

\begin{abstract} 
We consider empirical risk minimization of linear predictors 
with convex loss functions.
Such problems can be reformulated as convex-concave saddle point problems,
and thus are well suitable for primal-dual first-order algorithms.
However, primal-dual algorithms often require explicit strongly convex 
regularization in order to obtain fast linear convergence, and the required
dual proximal mapping may not admit closed-form or efficient solution. 
In this paper, we develop both batch and randomized primal-dual algorithms 
that can exploit strong convexity from data adaptively
and are capable of achieving linear convergence even without regularization.
We also present dual-free variants of the adaptive primal-dual algorithms 
that do not require computing the dual proximal mapping, 
which are especially suitable for logistic regression.
\end{abstract} 

\section{Introduction}

We consider the problem of regularized empirical risk minimization (ERM) 
of linear predictors.
Let $a_1, \ldots, a_n\in\R^d$ be the feature vectors of~$n$ data samples,
$\phi_i:\R\to\R$ be a convex loss function associated with 
the linear prediction $a_i^T x$, for $i=1,\ldots,n$,
and~$g:\R^d\to\R$ be a convex regularization function 
for the predictor~$x\in\R^d$.
ERM amounts to solving the following convex optimization problem:
\begin{align}\label{eqn:erm-primal}
  \min_{x\in\R^d} \quad 
  \textstyle
  \left\{ P(x) \eqdef \frac{1}{n} \sum_{i=1}^n \phi_i(a_i^{T}x)+ g(x) \right\} .
\end{align}
Examples of the above formulation 
include many well-known classification and regression problems.
For binary classification, 
each feature vector $a_i$ is associated with a label $b_i\in\{\pm 1\}$.
In particular logistic regression is obtained by setting 
$\phi_i(z)=\log(1+\exp(-b_i z))$. 
For linear regression problems, each feature vector $a_i$ is associated with 
a dependent variable $b_i\in\R$, and $\phi_i(z)=(1/2)(z-b_i)^2$.
Then we get ridge regression with $g(x) = (\lambda/2)\|x\|_2^2$,
and elastic net with $g(x)=\lambda_1 \|x\|_1+(\lambda_2/2)\|x\|_2^2$.

Let $A=[a_1,\ldots,a_n]^T$ be the data matrix.
Throughout this paper, we make the following assumptions:
\begin{assump}\label{asmp:erm}
The functions~$\phi_i$, $g$ and matrix~$A$ satisfy:
\vspace{-2ex}
\begin{itemize} \itemsep 0pt
  \item Each $\phi_i$ is $\delta$-strongly convex and $1/\gamma$-smooth
    where $\gamma>0$ and $\delta\geq 0$, and $\gamma\delta\leq 1$;
  \item $g$ is $\lambda$-strongly convex where $\lambda \geq 0$;
  \item $\lambda+\delta\mu^2>0$, where $\mu = \sqrt{\lambdamin(A^T A)}$.
\end{itemize}
\vspace{-2ex}
\end{assump}
The strong convexity and smoothness mentioned above are with respect to
the standard Euclidean norm, denoted as $\|x\| = \sqrt{x^T x}$.
(See, e.g., \citet[][Sections~2.1.1 and 2.1.3]{Nesterov2004book} for 
the exact definitions.)
Let $R=\max_i\{\|a_i\|\}$ and assuming $\lambda>0$, then $R^2/(\gamma\lambda)$
is a popular definition of condition number for analyzing complexities
of different algorithms.
The last condition above means that the primal objective function $P(x)$ 
is strongly convex, even if $\lambda=0$.

There have been extensive research activities in recent years on developing
efficiently algorithms for solving problem~\eqref{eqn:erm-primal}.
A broad class of randomized algorithms that exploit the finite sum 
structure in the ERM problem have emerged as very competitive
both in terms of theoretical complexity and practical performance.
They can be put into three categories: primal, dual, and primal-dual.

Primal randomized algorithms work with the ERM problem~\eqref{eqn:erm-primal}
directly. They are modern versions of randomized incremental gradient 
methods \citep[e.g.,][]{Bertsekas2012IncrementalChapter,Nedic} 
equipped with variance reduction
techniques. 
Each iteration of such algorithms only process one data point $a_i$
with complexity $O(d)$.
They includes 
SAG \citep{roux2012stochastic}, 
SAGA \citep{defazio2014saga}, and
SVRG \citep{johnson2013accelerating,xiaozhang2014proxsvrg},
which all achieve the iteration complexity 
$O\left( (n+R^2/(\gamma\lambda))\log(1/\epsilon) \right)$
to find an $\epsilon$-optimal solution.
In fact, they are capable of exploiting the strong convexity
from data, meaning that the condition number $R^2/(\gamma\lambda)$ in 
the complexity can be replaced by the more favorable one
$R^2/(\gamma(\lambda+\delta\mu^2/n))$.
This improvement can be achieved without explicit knowledge of $\mu$
from data.

Dual algorithms solve Fenchel dual of~\eqref{eqn:erm-primal}
by maximizing
\begin{equation}\label{eqn:erm-dual}
  \textstyle
  D(y) \eqdef 
    \frac{1}{n}\sum_{i=1}^n -\phi_i^*(y_i)
    - g^*\left(-\frac{1}{n} \sum_{i=1}^n y_i a_i\right)
\end{equation}
using randomized coordinate ascent algorithms.
(Here $\phi_i^*$ and $g^*$ denotes the conjugate functions of $\phi_i$ and $g$.)
They include SDCA \citep{SDCA}, \citet{Nesterov} and \citet{RichtarickTakac}.
They have the same complexity
$O\left( (n+R^2/(\gamma\lambda))\log(1/\epsilon) \right)$,
but are hard to exploit strong convexity from data.

Primal-dual algorithms solve the convex-concave saddle point problem
$\min_{x} \max_{y} \Lagr(x,y)$ where
\begin{align}\label{eqn:erm-saddle}
\textstyle
\Lagr(x,y) \eqdef 
\frac{1}{n}\sum_{i=1}^n 
\bigl( y_i \langle a_i, x \rangle - \phi^*_i(y_i) \bigr) + g(x) .
\end{align}
In particular, SPDC \citep{ZhangXiao2015SPDC} achieves an accelerated 
linear convergence rate with iteration complexity
$O\left((n+\sqrt{n}R/\sqrt{\gamma\lambda})\log(1/\epsilon)\right)$,
which is better than the aforementioned non-accelerated complexity
when $R^2/(\gamma\lambda)>n$.
\citet{LanZhou2015RPDG} developed dual-free variants of accelerated primal-dual
algorithms, but without considering the linear predictor structure in ERM. 
\citet{BalamuruganBach2016saddlepoint} extended SVRG and SAGA to solving 
saddle point problems.

Accelerated primal and dual randomized algorithms have also 
been developed.
\citet{Nesterov}, \citet{FercoqRichtarik2015} and \citet{LinLuXiao2015} 
developed accelerated coordinate gradient algorithms, 
which can be applied to solve the dual problem~\eqref{eqn:erm-dual}.
\citet{allen2016katyusha} developed an accelerated variant of SVRG.
Acceleration can also be obtained using the Catalyst framework
\citep{lin2015universal}.
They all achieve the same
$O\left((n+\sqrt{n}R/\sqrt{\gamma\lambda})\log(1/\epsilon)\right)$
complexity.
A common feature of accelerated algorithms 
is that they require good estimate of the strong convexity parameter.
This makes hard for them to exploit strong convexity from data because 
the minimum singular value~$\mu$ of the data matrix~$A$ is very hard to 
estimate in general.

In this paper, we show that primal-dual algorithms are capable of exploiting
strong convexity from data if the algorithm parameters (such as step sizes) 
are set appropriately. While these optimal setting depends on the knowledge 
of the convexity parameter~$\mu$ from the data, we develop adaptive variants
of primal-dual algorithms that can tune the parameter automatically.
Such adaptive schemes rely critically on the capability of evaluating
the primal-dual optimality gaps by primal-dual algorithms.

A major disadvantage of primal-dual algorithms is that the required
dual proximal mapping may not admit closed-form or efficient solution. 
We follow the approach of \citet{LanZhou2015RPDG} to
derive dual-free variants of the primal-dual algorithms customized for ERM
problems with the linear predictor structure, and show that they can also
exploit strong convexity from data with correct choices of parameters
or using an adaptation scheme.

\section{Batch primal-dual algorithms}
\label{sec:BPD}

Before diving into randomized primal-dual algorithms, we first consider
batch primal-dual algorithms, which exhibit similar properties as their
randomized variants.
To this end,
we consider a ``batch'' version of the ERM problem~\eqref{eqn:erm-primal},
\begin{equation}\label{eqn:batch-primal}
  \textstyle
  \min_{x\in\reals^d} ~ \bigl\{ P(x) \eqdef f(Ax) +  g(x) \bigr\}.
\end{equation}
where $A\in\reals^{n\times d}$, and make the following assumption: 
\begin{assump}\label{asmp:batch-PD}
The functions~$f$, $g$ and matrix~$A$ satisfy:
\vspace{-2ex}
\begin{itemize} \itemsep 0pt
  \item $f$ is $\delta$-strongly convex and $1/\gamma$-smooth
    where $\gamma>0$ and $\delta\geq 0$, and $\gamma\delta\leq 1$;
  \item $g$ is $\lambda$-strongly convex where $\lambda \geq 0$;
  \item $\lambda+\delta\mu^2>0$, where $\mu = \sqrt{\lambdamin(A^T A)}$.
\end{itemize}
\vspace{-2ex}
\end{assump}
For exact correspondence with problem~\eqref{eqn:erm-primal}, 
we have $f(z) = \frac{1}{n}\sum_{i=1}^n \phi_i(z_i)$ with $z_i=a_i^T x$.
Under Assumption~\ref{asmp:erm}, the function~$f(z)$ is  
$\delta/n$-strongly convex and $1/(n\gamma)$-smooth, 
and $f(Ax)$ is $\delta\mu^2/n$-strongly convex and $R^2/\gamma$-smooth.
However, such correspondences alone 
are not sufficient to exploit the structure of~\eqref{eqn:erm-primal},
i.e., substituting them into the batch algorithms of this section
will not produce the efficient algorithms for solving 
problem~\eqref{eqn:erm-primal} that we will present in 
Sections~\ref{sec:Ada-SPDC} and~\ref{sec:DF-SPDC}.
So we do not make such correspondences explicit in this section.
Rather, treat them as independent assumptions with the same notation.

\begin{algorithm}[tb]
   \caption{Batch Primal-Dual (BPD) Algorithm}
   \label{alg:BPD}
\begin{algorithmic}
   \REQUIRE parameters $\tau$, $\sigma$, $\theta$, 
          initial point $(\xtldini=\xini,\yini)$
   \FOR{$t=0,1,2,\ldots$}
   \STATE $\ytp = \prox_{\sigma f^*}\left(\yt+\sigma A \xtldt\right)$
   \STATE $\xtp = \prox_{\tau g}\left(\xt-\tau A^T \ytp\right)$
   \STATE $\xtldtp = \xtp + \theta(\xtp-\xt)$
   \ENDFOR
\end{algorithmic}
\end{algorithm}

Using conjugate functions, we can derive the dual of~\eqref{eqn:batch-primal} as
\begin{equation}\label{eqn:batch-dual}
  \textstyle
  \max_{y\in\reals^n} ~\bigl\{D(y) \eqdef -f^*(y) - g^*(-A^T y) \bigr\},
\end{equation}
and the convex-concave saddle point formulation is
\begin{equation} \label{eqn:batch-saddle}
  \min_{x\in\reals^d} \, \max_{y\in\reals^n} \,
  \bigl\{ \Lagr(x,y) \eqdef g(x) + y^T Ax -f^*(y) \bigr\}.
\end{equation}
We consider the primal-dual first-order algorithm proposed by
\citet{ChambollePock2011PrimalDual,ChambollePock2016Ergodic}
for solving the saddle point problem~\eqref{eqn:batch-saddle}, 
which is given as Algorithm~\ref{alg:BPD}.
Here we call it the batch primal-dual (BPD) algorithm.
Assuming that~$f$ is smooth and~$g$ is strongly convex,
\citet{ChambollePock2011PrimalDual,ChambollePock2016Ergodic}
showed that Algorithm~\ref{alg:BPD} achieves accelerated linear
convergence rate if $\lambda>0$.
However, they did not consider the case where additional or the sole source of
strong convexity comes from $f(Ax)$.

In the following theorem, we show how to set the parameters  
$\tau$, $\sigma$ and $\theta$ to exploit both sources of strong convexity
to achieve fast linear convergence.
 
\begin{theorem}\label{thm:BPD-convergence}
  Suppose Assumption~\ref{asmp:batch-PD} holds and $(\xopt,\yopt)$ is
  the unique saddle point of~$\Lagr$ defined in~\eqref{eqn:batch-saddle}.
  Let $L = \|A\| = \sqrt{\lambdamax(A^T A)}$.
  If we set the parameters in Algorithm~\ref{alg:BPD} as
\begin{equation}\label{eqn:batch-tau-sigma}
\textstyle 
\sigma=\frac{1}{L}\sqrt{\frac{\lambda+\delta\mu^2}{\gamma}}, \quad
\tau=\frac{1}{L}\sqrt{\frac{\gamma}{\lambda+\delta\mu^2}}, 
\end{equation}
and $\theta=\max\{\theta_x, \theta_y\}$ where
\begin{equation}\label{eqn:batch-theta}
\textstyle
  \theta_x = \Bigl(1-\frac{\delta}{(\delta+2\sigma)}\frac{\mu^2}{L^2}\Bigr)
             \frac{1}{1+\tau\lambda}, \quad
  \theta_y = \frac{1}{1+\sigma\gamma/2},
\end{equation}
then we have
\begin{align*}
\textstyle
\left(\frac{1}{2\tau}+\frac{\lambda}{2}\right)\|\xt-\xopt\|^2
  +\frac{\gamma}{4}\|\yt-\yopt\|^2   &\leq \theta^t C, \\
\textstyle
  \Lagr(\xt,\yopt) - \Lagr(\xopt,\yt) &\leq \theta^t C,
\end{align*}
where
$
C = \left(\frac{1}{2\tau}\!+\!\frac{\lambda}{2}\right)\|\xini-\xopt\|^2
  +\left(\frac{1}{2\sigma}\!+\!\frac{\gamma}{4}\right)
  \|\yini-\yopt\|^2.
$
\end{theorem}

The proof of Theorem~\ref{thm:BPD-convergence} is given in 
Appendices~\ref{sec:convergence_general} and~\ref{sec:analysis-Euclidean}.
Here we give a detailed analysis of the convergence rate.
Substituting $\sigma$ and $\tau$ in~\eqref{eqn:batch-tau-sigma} 
into the expressions for $\theta_y$ and $\theta_x$ in~\eqref{eqn:batch-theta}, 
and assuming $\gamma(\lambda+\delta\mu^2)\ll L^2$,
we have
\begin{align*}
\theta_x &\approx 
\textstyle
1 - \frac{\gamma\delta\mu^2}{L^2}
\Bigl(2\frac{\sqrt{\gamma(\lambda+\delta\mu^2)}}{L} + \gamma\delta\Bigr)^{-1}
  - \frac{\lambda}{L}\sqrt{\frac{\gamma}{\lambda+\delta\mu^2}}, \\
\theta_y &= 
\textstyle
\frac{1}{1+\sqrt{\gamma(\lambda+\delta\mu^2)}/(2L)}
  \approx 1 - \frac{\sqrt{\gamma(\lambda+\delta\mu^2)}}{2L}.
\end{align*}
Since the overall condition number of the problem is 
$\frac{L^2}{\gamma(\lambda+\delta\mu^2)}$,
it is clear that $\theta_y$ is an accelerated convergence rate.
Next we examine $\theta_x$ in two special cases.

\vspace{-1ex}
\paragraph{The case of $\delta\mu^2=0$ but $\lambda>0$.}
In this case, we have
$\tau = \frac{1}{L}\sqrt{\frac{\gamma}{\lambda}}$ and
$\sigma = \frac{1}{L}\sqrt{\frac{\lambda}{\gamma}}$,
and thus
\begin{align*}
  \theta_x \textstyle 
  \!=\!\frac{1}{1+\sqrt{\gamma\lambda}/L} 
  \approx 1 \!-\! \frac{\sqrt{\gamma\lambda}}{L}, \quad
  \theta_y 
  \!=\!\frac{1}{1+\sqrt{\gamma\lambda}/(2L)} 
  \approx 1\! -\! \frac{\sqrt{\gamma\lambda}}{2L}.
\end{align*}
Therefore we have
$
  \theta = \max\{\theta_x, \theta_y\}
  \approx 1 - \frac{\sqrt{\lambda\gamma}}{2L} .
$
This indeed is an accelerated convergence rate, recovering the result of
\citet{ChambollePock2011PrimalDual,ChambollePock2016Ergodic}.

\vspace{-1ex}
\paragraph{The case of $\lambda=0$ but $\delta\mu^2>0$.}
In this case, we have
  $\tau = \frac{1}{L\mu}\sqrt{\frac{\gamma}{\delta}}$ and
  $\sigma = \frac{\mu}{L}\sqrt{\frac{\delta}{\gamma}}$,
and
\begin{align*}
  \theta_x = 
  \textstyle
  1 - \frac{\gamma\delta\mu^2}{L^2}\cdot
  \frac{1}{2\sqrt{\gamma\delta}\mu/L+\gamma\delta}, \quad
  \theta_y
  \approx 1 - \frac{\sqrt{\gamma\delta}\mu}{2L}.
\end{align*}
Notice that 
$\frac{1}{\gamma\delta}\frac{L^2}{\mu^2}$ is the condition number of
$f(Ax)$.
Next we assume $\mu\ll L$ and examine how $\theta_x$ varies with
$\gamma\delta$.
\begin{itemize}\itemsep 0pt
  \item If $\gamma\delta\approx\frac{\mu^2}{L^2}$, meaning $f$ is 
    badly conditioned, then 
    \[
      \textstyle
      \theta_x \approx 1 - \frac{\gamma\delta\mu^2}{L^2}\cdot
      \frac{1}{3\sqrt{\gamma\delta}\mu/L}
      = 1 - \frac{\sqrt{\gamma\delta}\mu}{3L}.
    \]
    Because
    the overall condition number is $\frac{1}{\gamma\delta}\frac{L^2}{\mu^2}$,
    this is an accelerated linear rate, and so is $\theta=\max\{\theta_x,\theta_y\}$.
  \item If $\gamma\delta\approx\frac{\mu}{L}$, meaning $f$ is mildly conditioned, then
    \[
      \textstyle
      \theta_x \approx 1 - \frac{\mu^3}{L^3}
      \frac{1}{2\left(\mu/L\right)^{3/2} + \mu/L}
      \approx 1 - \frac{\mu^2}{L^2} .
    \]
    This represents a half-accelerated rate, because
    the overall condition number is 
    $\frac{1}{\gamma\delta}\frac{L^2}{\mu^2}\approx\frac{L^3}{\mu^3}$.
  \item If $\gamma\delta=1$, i.e., $f$ is a simple quadratic function, then
    \[
      \textstyle
      \theta_x \approx 1 - \frac{\mu^2}{L^2}\frac{1}{2\mu/L + 1}
      \approx 1 - \frac{\mu^2}{L^2} .
    \]
    This rate does not have acceleration, because
    the overall condition number is 
    $\frac{1}{\gamma\delta}\frac{L^2}{\mu^2}\approx\frac{L^2}{\mu^2}$.
\end{itemize}
In summary, the extent of acceleration in the dominating factor $\theta_x$ 
(which determines $\theta$) depends on the relative size of $\gamma\delta$
and $\mu^2/L^2$, i.e., the relative conditioning between the function~$f$
and the matrix~$A$.
In general, we have full acceleration if $\gamma\delta\leq\mu^2/L^2$.
The theory predicts that the acceleration degrades as the function~$f$ gets
better conditioned.
However, in our numerical experiments, 
we often observe acceleration even if $\gamma\delta$ gets closer to 1.

As explained in \citet{ChambollePock2011PrimalDual}, Algorithm~\ref{alg:BPD}
is equivalent to a preconditioned ADMM.
\citet{deng2016global} characterized conditions for ADMM to obtain
linear convergence without assuming both parts of the objective
function being strongly convex, but they 
did not derive convergence rate for this case.

\subsection{Adaptive batch primal-dual algorithms}

\begin{algorithm}[tb]
   \caption{Adaptive Batch Primal-Dual (Ada-BPD)}
   \label{alg:Ada-BPD}
\begin{algorithmic}
  \REQUIRE problem constants $\lambda$, $\gamma$, $\delta$, $L$ 
           and $\hat{\mu}>0$, initial
   \STATE \qquad point $(\xini,\yini)$, and adaptation period $T$.
   \vspace{0.3ex}
   \STATE Compute $\sigma$, $\tau$, and $\theta$ as 
          in~\eqref{eqn:batch-tau-sigma} and~\eqref{eqn:batch-theta} 
          using $\mu=\hat{\mu}$
   \FOR{$t=0,1,2,\ldots$}
   \vspace{0.5ex}
   \STATE $\ytp = \prox_{\sigma f^*}\left(\yt+\sigma A \xtldt\right)$
   \STATE $\xtp = \prox_{\tau g}\left(\xt-\tau A^T \ytp\right)$
   \STATE $\xtldtp = \xtp + \theta(\xtp-\xt)$
   \vspace{0.5ex}
   \IF{$\mbox{mod}(t+1,T)==0$}
   \STATE $(\sigma,\tau,\theta)=\mbox{BPD-Adapt}
           \left(\{P^{(s)},D^{(s)}\}_{s=t-T}^{t+1}\right)$
   \ENDIF
   \ENDFOR
\end{algorithmic}
\end{algorithm}

In practice, it is often very hard to obtain good estimate of
the problem-dependent constants, especially 
$\mu=\sqrt{\lambdamin(A^T A)}$, in order to apply the algorithmic parameters
specified in Theorem~\ref{thm:BPD-convergence}.
Here we explore heuristics that can enable adaptive tuning of such parameters,
which often lead to much improved performance in practice.

A key observation is that the convergence rate of the BPD algorithm
changes monotonically with the overall strong convexity parameter 
$\lambda+\delta\mu^2$, regardless of the extent of acceleration.
In other words, the larger $\lambda+\delta\mu^2$ is, the faster the convergence.
Therefore, if we can monitor the progress of the convergence and compare it
with the predicted convergence rate in Theorem~\ref{thm:BPD-convergence},
then we can adjust the algorithmic parameters to exploit the fastest possible
convergence.
More specifically, if the observed convergence is slower than the predicted
convergence rate, then we should reduce the estimate of $\mu$;
if the observed convergence is better than the predicted rate, then we can
try to increase $\mu$ for even faster convergence.

\begin{algorithm}[tb]
   \caption{BPD-Adapt (simple heuristic)}
   \label{alg:BPD-Adapt-simple}
\begin{algorithmic}
  \REQUIRE previous estimate $\hat{\mu}$, adaption period $T$, primal and%
  \STATE \qquad dual objective values $\{P^{(s)},D^{(s)}\}_{s=t-T}^{t}$
   \vspace{0.5ex}
   \IF{$P^{(t)}-D^{(t)}<\theta^T(P^{(t-T)}-D^{(t-T)})$}
   \STATE $\hat{\mu} := \sqrt{2}\hat{\mu}$
   \ELSE
   \STATE $\hat{\mu} := \hat{\mu}/\sqrt{2}$
   \ENDIF
   \STATE Compute $\sigma$, $\tau$, and $\theta$ as 
          in~\eqref{eqn:batch-tau-sigma} and~\eqref{eqn:batch-theta} 
          using $\mu=\hat{\mu}$
   \ENSURE new parameters $(\sigma, \tau, \theta)$
\end{algorithmic}
\end{algorithm}

We formalize the above reasoning in an Adaptive BPD (Ada-BPD) algorithm
described in Algorithm~\ref{alg:Ada-BPD}.
This algorithm maintains an estimate $\hat{\mu}$ of the true constant~$\mu$,
and adjust it every~$T$ iterations.
We use $P^{(t)}$ and $D^{(t)}$ to represent the primal and dual objective
values at $P(\xt)$ and $D(\yt)$, respectively.
We give two implementations of the tuning procedure BPD-Adapt:
\vspace{-1ex}
\begin{itemize} \itemsep 0pt
  \item
Algorithm~\ref{alg:BPD-Adapt-simple} is a simple heuristic for
    tuning the estimate $\hat{\mu}$, where the increasing and decreasing
    factor $\sqrt{2}$ can be changed to other values larger than 1;
  \item 
Algorithm~\ref{alg:BPD-Adapt-robust} is a more robust heuristic.
    It does not rely on the specific convergence rate $\theta$ established
    in Theorem~\ref{thm:BPD-convergence}. Instead, it simply compares the 
    current estimate of objective reduction rate $\hat{\rho}$ with the 
    previous estimate $\rho$ ($\approx\theta^T$). 
    It also specifies a non-tuning range of changes in~$\rho$, specified
    by the interval $[\underline{c}, \overline{c}]$.
\end{itemize}
\vspace{-1ex}
One can also devise more sophisticated schemes; e.g., 
if we estimate that $\delta\mu^2 < \lambda$, then no more tuning is necessary.

The capability of accessing both the primal and dual objective values
allows primal-dual algorithms to have good estimate of the convergence rate,
which enables effective tuning heuristics. 
Automatic tuning of primal-dual algorithms have also been studied by, e.g.,
\citet{malitsky2016first} and
\citet{goldstein2013adaptive},
but with different goals.

Finally, we note that Theorem~\ref{thm:BPD-convergence} only establishes
convergence rate for the distance to the optimal point and the quantity
$\Lagr(\xt,\yopt)-\Lagr(\xopt, \yt)$, which is not quite the duality
gap $P(\xt)-D(\yt)$.
Nevertheless, same convergence rate can also be established for the duality 
gap \citep[see][Section~2.2]{ZhangXiao2015SPDC}, which can be used to
better justify the adaption procedure.

\begin{algorithm}[tb]
   \caption{BPD-Adapt (robust heuristic)}
   \label{alg:BPD-Adapt-robust}
\begin{algorithmic}
  \REQUIRE previous rate estimate $\rho>0$, $\Delta=\delta\hat{\mu}^2$, period $T$, 
  \STATE \qquad constants $\underline{c}<1$ and $\overline{c}>1$, and $\{P^{(s)},D^{(s)}\}_{s=t-T}^{t}$\!\!\!%
   \vspace{0.5ex}
   \STATE Compute new rate estimate $\hat{\rho} = \frac{P^{(t)}-D^{(t)}}{P^{(t-T)}-D^{(t-T)}}$
   \vspace{0.3ex}
   \IF{$\hat{\rho} \leq \underline{c}\,\rho$}
   \STATE $\Delta := 2\Delta$, \quad~ $\rho:=\hat{\rho}$
   \ELSIF {$\hat{\rho} \geq \overline{c}\,\rho$}
   \STATE $\Delta := \Delta/2$, \quad $\rho:=\hat{\rho}$
   \ELSE
   \STATE $\Delta:=\Delta$
   \ENDIF
   \vspace{0.5ex}
   \STATE $\sigma=\frac{1}{L}\sqrt{\frac{\lambda+\Delta}{\gamma}}$,
          \quad $\tau=\frac{1}{L}\sqrt{\frac{\gamma}{\lambda+\Delta}}$
   \vspace{0.3ex}
   \STATE Compute $\theta$ using~\eqref{eqn:batch-theta} or set $\theta=1$
   \vspace{0.5ex}
   \ENSURE new parameters $(\sigma, \tau, \theta)$
\end{algorithmic}
\end{algorithm}

\section{Randomized primal-dual algorithm}
\label{sec:Ada-SPDC}
 
In this section, we come back to the ERM problem~\eqref{eqn:erm-primal},
which have a finite sum structure that allows the development of
randomized primal-dual algorithms.
In particular, we extend the stochastic primal-dual coordinate (SPDC)
algorithm \citep{ZhangXiao2015SPDC} to exploit the strong convexity
from data in order to achieve faster convergence rate.

First, we show that, by setting algorithmic parameters appropriately, 
the original SPDC algorithm may directly benefit from strong convexity
from the loss function. 
We note that
the SPDC algorithm is a special case of the Adaptive SPDC (Ada-SPDC) algorithm
presented in Algorithm~\ref{alg:Ada-SPDC}, by setting the adaption period
$T=\infty$ (not performing any adaption).
The following theorem is proved in Appendix~\ref{sec:proof_spdc}.

\begin{theorem}\label{thm:SPDC-convergence}
  Suppose Assumption~\ref{asmp:erm} holds. Let $(\xopt,\yopt)$ be
  the saddle point of the function~$\Lagr$ defined in~\eqref{eqn:erm-saddle}, 
  and $R=\max\{\|a_1\|,\ldots,\|a_n\|\}$.
  If we set $T=\infty$ in Algorithm~\ref{alg:Ada-SPDC} (no adaption) 
  and let
\begin{equation}\label{eqn:SPDC-tau-sigma}
\textstyle 
\tau=\frac{1}{4R}\sqrt{\frac{\gamma}{n\lambda+\delta\mu^2}}, \quad 
\sigma=\frac{1}{4R}\sqrt{\frac{n\lambda+\delta\mu^2}{\gamma}},
\end{equation}
and $\theta=\max\{\theta_x, \theta_y\}$ where
\begin{equation}\label{eqn:SPDC-theta}
\textstyle
\!\!\! \theta_x = \Bigl(1\!-\!\frac{\tau\sigma\delta\mu^2}{2n(\sigma+4\delta)}\Bigr)
             \frac{1}{1+\tau\lambda}, ~~
             \theta_y = \frac{1+((n-1)/n)\sigma\gamma/2}{1+\sigma\gamma/2},
\end{equation}
then we have
\begin{align*}
\textstyle
\left(\frac{1}{2\tau}+\frac{\lambda}{2}\right)\E\bigl[\|\xt-\xopt\|^2\bigr]
+\frac{\gamma}{4}\E\bigl[\|\yt-\yopt\|^2\bigr]  &\leq \theta^t C, \\
\textstyle
\E\left[\Lagr(\xt,\yopt) - \Lagr(\xopt,\yt)\right] &\leq \theta^t C,
\end{align*}
where
$
C = \left(\frac{1}{2\tau}\!+\!\frac{\lambda}{2}\right)\|\xini-\xopt\|^2
  +\left(\frac{1}{2\sigma}\!+\!\frac{\gamma}{4}\right)
  \|\yini-\yopt\|^2.
$
The expectation $\E[\cdot]$ is taken
with respect to the history of random indices drawn at each iteration.
\end{theorem}

\begin{algorithm}[tb]
   \caption{Adaptive SPDC (Ada-SPDC)}
   \label{alg:Ada-SPDC}
\begin{algorithmic}
   \REQUIRE parameters $\sigma$, $\tau$, $\theta>0$, 
          initial point $(\xini,\yini)$,
   \STATE \qquad and adaptation period $T$.
   \vspace{0.5ex}
   \STATE Set $\xtld^{(0)} = \xini$
   \FOR{$t=0,1,2,\ldots$}
   \STATE pick $k\in\{1,\ldots,n\}$ uniformly at random
   \FOR{$i\in\{1,\ldots,n\}$}
   \IF{$i==k$}
   \STATE $\yktp = \prox_{\sigma\phi_k^*}\!\left(\ykt+\sigma a_k^T\xtldt\right)$
   \ELSE
   \STATE $\yitp = \yit$
   \ENDIF
   \ENDFOR
   \vspace{1ex}
   \STATE $\xtp = \prox_{\tau g}\!\left(\xt\!-\tau\bigl(\ut\!+(\yktp\!\!-\!\ykt)a_k\bigr)\right)$\\[0.5ex]
   \STATE $\utp = \ut + \frac{1}{n} (\yktp-\ykt)a_k$\\[0.5ex]
   \STATE $\xtldtp = \xtp + \theta(\xtp-\xt)$
   \vspace{1ex}
   \IF{$\mbox{mod}(t+1,T\cdot n)=0$}
   \STATE $(\tau,\sigma,\theta)=\mbox{SPDC-Adapt}\bigl(\{P^{(t-sn)},D^{(t-sn)}\}_{s=0}^{T}\bigr)$
   \ENDIF
   \ENDFOR
\end{algorithmic}
\end{algorithm}

Below we give a detailed discussion on the expected convergence rate
established in Theorem~\ref{thm:SPDC-convergence}.

\vspace{-1ex}
\paragraph{The cases of $\sigma \mu^2 = 0$ but $\lambda > 0$.}
In this case we have
$\tau = \frac{1}{4R} \sqrt{ \frac{\gamma}{n \lambda } }$ and
$\sigma = \frac{1}{4R} \sqrt{ \frac{n \lambda }{\gamma  }}$, and
\begin{align*}
\theta_x &= \textstyle 
\frac{1}{1+\tau\lambda} = 1 - \frac{1}{1 + 4 R \sqrt{n/(\lambda \gamma)}}, \\
\theta_y &= \textstyle 
\frac{1+((n-1)/n)\sigma\gamma/2}{1+\sigma\gamma/2}
= 1 - \frac{1}{n + 8 R \sqrt{n/(\lambda \gamma)}}.
\end{align*}
Hence $\theta=\theta_y$.
These recover the parameters and convergence rate
of the standard SPDC \citep{ZhangXiao2015SPDC}.

\vspace{-1ex}
\paragraph{The cases of $\sigma \mu^2 > 0$ but $\lambda = 0$.}
In this case we have
$\tau = \frac{1}{4R \mu} \sqrt{ \frac{\gamma}{ \delta } }$
and $\sigma = \frac{\mu}{4R} \sqrt{ \frac{ \delta }{\gamma  }}$, and
\begin{align*}
\theta_x &=
\textstyle
1 - \frac{\tau \sigma \delta \mu^2}{2 n(\sigma + 4 \delta )} 
= 1 - \frac{\gamma \delta \mu^2}{32 n R^2} \cdot \frac{1}{ \sqrt{\gamma\delta}\mu/(4R) + 4  \gamma \delta}. \\
\theta_y &= 
\textstyle
1 - \frac{1}{n + 8nR/(\mu \sqrt{\gamma \delta})}
\approx 1 - \frac{\sqrt{\gamma\delta}\mu}{8nR}\left(1+\frac{\sqrt{\gamma\delta}\mu}{8R}\right)^{-1}.
\end{align*}
Since the objective is $R^2/\gamma$-smooth 
and $\delta \mu^2/n$-strongly convex, 
$\theta_y$ is an accelerated rate if $\frac{\sqrt{\gamma\delta}\mu}{8R}\ll 1$
(otherwise $\theta_y\approx1-\frac{1}{n}$). 
For $\theta_x$, we consider different situations:
\vspace{-1ex}
\begin{itemize}\itemsep 0pt
\item If $\mu \geq R$, then we have 
$
\theta_x \approx 1\! -\! \frac{\sqrt{\gamma \delta } \mu }{ n R },
$ which is an accelerated rate. 
So is $\theta=\max\{\theta_x,\theta_y\}$. 
\item If $\mu < R$ and $\gamma \delta \approx \frac{\mu^2}{ R^2}$, then  
$
\theta_x \approx 1\! -\! \frac{\sqrt{\gamma \delta } \mu }{ n R },
$
which represents accelerated rate.
The iteration complexity of SPDC is 
$\widetilde{O}(\frac{n R}{\mu \sqrt{\gamma \delta}})$, which is better than
that of SVRG in this case, which is
$\widetilde{O}(\frac{ n R^2}{ \gamma \delta \mu^2 })$.
\item If $\mu < R$ and $\gamma \delta \approx \frac{\mu}{ R}$, then we get
$
\theta_x \approx 1 - \frac{\mu^2}{ n R^2}.
$
This is a half-accelerated rate, because in this case SVRG would require 
$\widetilde{O}(\frac{n R^3}{\mu^3})$ iterations, 
while iteration complexity here is 
$ \widetilde{O}(\frac{n R^2}{\mu^2})$.
\item If $\mu < R$ and $\gamma \delta \approx 1$, 
  meaning the $\phi_i$'s are well conditioned, then we get
$
\theta_x \approx 1 - \frac{ \gamma \delta \mu^2 }{ n R^2} \approx 1 - \frac{  \mu^2 }{ n R^2},
$
which is a non-accelerated rate. 
The corresponding iteration complexity is the same as SVRG.
\end{itemize}

\subsection{Parameter adaptation for SPDC}

The SPDC-Adapt procedure called in Algorithm~\ref{alg:Ada-SPDC}
follows the same logics as the batch adaption schemes in
Algorithms~\ref{alg:BPD-Adapt-simple} and~\ref{alg:BPD-Adapt-robust},
and we omit the details here.
One thing we emphasize here is that the adaptation period $T$ is in terms of
epochs, or number of passes over the data.
In addition, we only compute the primal and dual objective values after
each pass or every few passes, 
because computing them exactly usually need to take a full pass of the data.

Another important issue is that, unlike the batch case where the duality gap
usually decreases monotonically, the duality gap for randomized algorithms
can fluctuate wildly. 
So instead of using only the two end values $P^{(t-Tn)}-D^{(t-Tn)}$ and
$P^{(t)}-D^{(t)}$, we can use more points to 
estimate the convergence rate through a linear regression.
Suppose the primal-dual values at the end of each past $T+1$ passes are
\[
  \{P(0),D(0)\}, \{P(1),D(1)\}, \ldots, \{P(T), D(T)\},
\]
and we need to estimate $\rho$ (rate per pass) such that
\[
  P(t)-D(t) \approx \rho^t \bigl(P(0)-D(0)\bigr),\quad
  t=1,\ldots,T.
\]
We can turn it into a linear regression problem after taking logarithm
and obtain the estimate $\hat{\rho}$ through
\[
  \textstyle
  \log(\hat{\rho}) = \frac{1}{1^2+2^2+\cdots +T^2}
  \sum_{t=1}^T  t \log\frac{P(t)-D(t)}{P(0)-D(0)} .
\]
The rest of the adaption procedure can follow the robust scheme in 
Algorithm~\ref{alg:BPD-Adapt-robust}.
In practice, we can compute the primal-dual values more sporadically, 
say every few passes,
and modify the regression accordingly.

\section{Dual-free Primal-dual algorithms}
\label{sec:DF-primal-dual}

Compared with primal algorithms,
one major disadvantage of primal-dual algorithms is the requirement of
computing the proximal mapping of the dual function $f^*$ or $\phi_i^*$,
which may not admit closed-formed solution or efficient computation.
This is especially the case for logistic regression, 
one of the most popular loss functions used in classification.

\citet{LanZhou2015RPDG} developed ``dual-free'' variants of primal-dual 
algorithms that avoid computing the dual proximal mapping.
Their main technique is to replace the Euclidean distance in the dual proximal
mapping with a Bregman divergence defined over the dual loss function itself.
We show how to apply this approach to solve the structured ERM problems
considered in this paper.
They can also exploit strong convexity from data
if the algorithmic parameters are set appropriately or adapted automatically.

\subsection{Dual-free BPD algorithm}
\label{sec:DF-BPD}

\begin{algorithm}[tb]
   \caption{Dual-Free BPD Algorithm}
   \label{alg:DF-BPD}
\begin{algorithmic}
   \REQUIRE parameters $\sigma$, $\tau$, $\theta>0$, 
          initial point $(\xini,\yini)$ 
   \vspace{0.5ex}
   \STATE Set $\xtld^{(0)} = \xini$ and $\vini=(f^*)'(\yini)$
   \FOR{$t=0,1,2,\ldots$}
   \STATE $\vtp = \frac{\vt+\sigma A\xtldt}{1+\sigma}$, \quad $\ytp = f'(\vtp)$
   \STATE $\xtp = \prox_{\tau g}\left(\xt-\tau A^T \ytp\right)$
   \STATE $\xtldtp = \xtp + \theta(\xtp-\xt)$
   \ENDFOR
\end{algorithmic}
\end{algorithm}

First, we consider the batch setting.
We replace the dual proximal mapping 
(computing $\ytp$) in Algorithm~\ref{alg:BPD} with
\begin{equation}\label{eqn:BPD-Bregman-y-update}
  \ytp\! =\! \argmin_{y} 
  \textstyle
\!\left\{ f^*(y)\!-\! y^T A\xtldt \!+\! \frac{1}{\sigma} \Bdiv(y,\yt) \right\},
\end{equation}
where $\Bdiv$ is the Bregman divergence of a strictly convex kernel 
function~$h$, defined as
\[
      \Bdiv_h(y,\yt) = h(y) - h(\yt) - \langle \nabla h(\yt), y-\yt\rangle. 
\]
Algorithm~\ref{alg:BPD} is obtained in the Euclidean setting with
$h(y)=\frac{1}{2}\|y\|^2$ and $\Bdiv(y,\yt)=\frac{1}{2}\|y-\yt\|^2$.
While our convergence results would apply for arbitrary
Bregman divergence, we only focus on the case of using $f^*$ itself as the
kernel, because this allows us to compute $\ytp$ 
in~\eqref{eqn:BPD-Bregman-y-update} very efficiently.
The following lemma explains the details
\citep[Cf.][Lemma~1]{LanZhou2015RPDG}.

\begin{lemma}\label{lem:batch-DF}
Let the kernel $h\equiv f^*$ in the Bregman divergence $\Bdiv$.
If we construct a sequence of vectors $\{\vt\}$ such that
$\vini = (f^*)'(\yini)$ and for all $t\geq0$,
\begin{equation}\label{eqn:DF-v-update}
  \textstyle
  \vtp = \frac{\vt+\sigma A\xtldt}{1+\sigma},
\end{equation}
then the solution to problem~\eqref{eqn:BPD-Bregman-y-update} is
$
  \ytp = f'(\vtp).
$
\end{lemma}
\begin{proof}
Suppose $\vt=(f^*)'(\yt)$ (true for $t=0$), then
\[
  \Bdiv(y,\yt) = f^*(y)-f^*(\yt)-\vt^T(y-\yt).
\]
The solution to~\eqref{eqn:BPD-Bregman-y-update} can be written as
\begin{align*}
\ytp 
\!&= \argmin_y \textstyle \Bigl\{f^*(y)\!-\!y^TA\xtldt\!+\!\frac{1}{\sigma}\bigl(f^*(y)\!-\!\vt^T y\bigr)\!\Bigr\}\\
\!&= \argmin_y \textstyle \Bigl\{\left(1+\frac{1}{\sigma}\right)\!f^*(y)-\left(A\xtldt+\frac{1}{\sigma}\vt\right)^T\! y\Bigr\}\\
\!&= \argmax_y \textstyle \Bigl\{\left(\frac{\vt+\sigma A\xtldt}{1+\sigma}\right)^T y - f^*(y)\Bigr\}\\
\!&= \argmax_y \textstyle \left\{\vtp^T y - f^*(y)\right\} = f'(\vtp),
\end{align*}
where in the last equality we used the property of conjugate function
when~$f$ is strongly convex and smooth.
Moreover, 
\[
  \vtp = (f')^{-1}(\ytp) = (f^*)'(\ytp),
\]
which completes the proof.
\end{proof}
According to Lemma~\ref{lem:batch-DF}, we only need to provide initial 
points such that $\vini=(f^*)'(\yini)$ is easy to compute. 
We do not need to compute $(f^*)'(\yt)$ directly for any $t>0$, because
it is can be updated as $\vt$ in~\eqref{eqn:DF-v-update}.
Consequently, we can update $\yt$ in the BPD algorithm using the
gradient $f'(\vt)$, without the need of dual proximal mapping.
The resulting dual-free algorithm is given in Algorithm~\ref{alg:DF-BPD}.
 
\citet{LanZhou2015RPDG} considered a general setting which 
does not possess the linear predictor structure we focus on in this paper,
and assumed that only the regularization~$g$ is strongly convex. 
Our following result shows that dual-free primal-dual algorithms
can also exploit strong convexity from data
with appropriate algorithmic parameters.

\begin{theorem}\label{thm:DF-BPD-convergence}
  Suppose Assumption~\ref{asmp:batch-PD} holds and let $(\xopt,\yopt)$ be
  the unique saddle point of~$\Lagr$ defined in~\eqref{eqn:batch-saddle}.
  If we set the parameters in Algorithm~\ref{alg:DF-BPD} as
\begin{equation}\label{eqn:Bregman-BPD-tau-sigma}
\textstyle 
\tau=\frac{1}{L}\sqrt{\frac{\gamma}{\lambda+\delta\mu^2}}, \quad
\sigma=\frac{1}{L}\sqrt{\gamma(\lambda+\delta\mu^2)}, 
\end{equation}
and $\theta=\max\{\theta_x, \theta_y\}$ where
\begin{equation}\label{eqn:Bregman-BPD-theta}
\textstyle
  \theta_x = \left(1-\frac{\tau\sigma\delta\mu^2}{(4+2\sigma)}\right)
             \frac{1}{1+\tau\lambda}, \quad
  \theta_y = \frac{1}{1+\sigma/2},
\end{equation}
then we have
\begin{align*}
\textstyle
\left(\frac{1}{2\tau}+\frac{\lambda}{2}\right)\|\xt-\xopt\|^2
  +\frac{1}{2}\Bdiv(\yopt,\yt)  &\leq \theta^t C, \\
\textstyle
  \Lagr(\xt,\yopt) - \Lagr(\xopt,\yt) &\leq \theta^t C,
\end{align*}
where
$
C = \left(\frac{1}{2\tau}\!+\!\frac{\lambda}{2}\right)\|\xini-\xopt\|^2
  +\left(\frac{1}{\sigma}\!+\!\frac{1}{2}\right)\Bdiv(\yopt,\yini).
$
\end{theorem}

Theorem~\ref{thm:DF-BPD-convergence} is proved in 
Appendices~\ref{sec:convergence_general} and~\ref{sec:proof_dfbpd}.
Assuming $\gamma(\lambda+\delta\mu^2)\ll L^2$, we have
\begin{align*}
  \textstyle
  \theta_x \approx 1-\frac{\gamma\delta\mu^2}{16 L^2}
  - \frac{\lambda}{2L}\sqrt{\frac{\gamma}{\lambda+\delta\mu^2}}, \quad
  \theta_y  
  \approx 1 - \frac{\sqrt{\gamma(\lambda+\delta\mu^2)}}{4L}.
\end{align*}
Again, we gain insights by consider the special cases:
\vspace{-1ex}
\begin{itemize} \itemsep 0pt
  \item If $\delta\mu^2 = 0$ and $\lambda >0$, then 
   $ 
      \theta_y \approx 1 - \frac{\sqrt{\gamma\lambda}}{4 L}
   $
   and
   $
      \theta_x \approx 1 - \frac{\sqrt{\gamma\lambda}}{2 L}.
   $ 
   So $\theta=\max\{\theta_x,\theta_y\}$ is an accelerated rate.
  \item If $\delta\mu^2>0$ and $\lambda=0$, then 
   $ 
      \theta_y \approx 1 - \frac{\sqrt{\gamma\delta\mu^2}}{4 L}
   $
   and
   $
      \theta_x \approx 1 - \frac{\gamma\delta\mu^2}{16 L^2}
   $.
   Thus
    $\theta=\max\{\theta_x,\theta_y\}\approx 1-\frac{\gamma\delta\mu^2}{16L^2}$
    is not accelerated.
    Notice that this conclusion does not depends on the relative size of 
    $\gamma\delta$ and $\mu^2/L^2$, 
    and this is the major difference from
    the Euclidean case discussed in Section~\ref{sec:BPD}.
\end{itemize}
If both $\delta\mu^2>0$ and $\lambda>0$, then the extent of acceleration
    depends on their relative size. 
    If $\lambda$ is on the same order as $\delta\mu^2$ or larger, then
    accelerated rate is obtained. 
    If $\lambda$ is much smaller than $\delta\mu^2$, then the theory predicts
    no acceleration.

\subsection{Dual-free SPDC algorithm}
\label{sec:DF-SPDC}

The same approach can be applied to derive an Dual-free SPDC algorithm,
which is described in Algorithm~\ref{alg:ADF-SPDC}.
It also includes a parameter adaption procedure, so we call it the 
adaptive dual-free SPDC (ADF-SPDC) algorithm.
On related work,
\citet{ShalevZhang2016accSDCAinMathProg} and \cite{shalev2016sdca} 
introduced dual-free SDCA.

\begin{algorithm}[tb]
   \caption{Adaptive Dual-Free SPDC (ADF-SPDC)}
   \label{alg:ADF-SPDC}
\begin{algorithmic}
   \REQUIRE parameters $\sigma$, $\tau$, $\theta>0$, 
          initial point $(\xini,\yini)$,
   \STATE \qquad and adaptation period $T$.
   \vspace{0.5ex}
   \STATE Set $\xtld^{(0)} = \xini$ and $\viini=(\phi_i^*)'(\yiini)$ for $i=1,\ldots,n$
   \FOR{$t=0,1,2,\ldots$}
   \STATE pick $k\in\{1,\ldots,n\}$ uniformly at random
   \FOR{$i\in\{1,\ldots,n\}$}
   \IF{$i==k$} 
   \vspace{0.5ex}
   \STATE $\vktp = \frac{\vkt + \sigma a_k^T \xtldt}{1+\sigma}$,
        ~~$\yktp = \phi'_k(\vktp)$
   \ELSE
   \STATE $\vitp = \vit$,
        ~~$\yitp = \yit$
   \ENDIF
   \ENDFOR
   \vspace{1ex}
   \STATE $\xtp = \prox_{\tau g}\!\left(\xt\!-\tau\bigl(\ut\!+(\yktp\!\!-\!\ykt)a_k\bigr)\right)$\\[0.5ex]
   \STATE $\utp = \ut + \frac{1}{n} (\yktp-\ykt)a_k$\\[0.5ex]
   \STATE $\xtldtp = \xtp + \theta(\xtp-\xt)$
   \vspace{1ex}
   \IF{$\mbox{mod}(t+1,T\cdot n)=0$}
   \STATE $(\tau,\sigma,\theta)=\mbox{SPDC-Adapt}\bigl(\{P^{(t-sn)},D^{(t-sn)}\}_{s=0}^{T}\bigr)$
   \ENDIF
   \ENDFOR
\end{algorithmic}
\end{algorithm}

The following theorem characterizes the choice of algorithmic parameters
that can exploit strong convexity from data to achieve linear
convergence (proof given in Appendix~\ref{sec:proof_adf_spdc}).

\begin{theorem}\label{thm:DF-SPDC-convergence}
  Suppose Assumption~\ref{asmp:erm} holds. Let $(\xopt,\yopt)$ be
  the saddle point of~$\Lagr$ defined in~\eqref{eqn:erm-saddle}
  and $R=\max\{\|a_1\|,\ldots,\|a_n\|\}$.
  If we set $T=\infty$ in Algorithm~\ref{alg:ADF-SPDC} (non adaption) 
  and let
\begin{equation}\label{eqn:DF-SPDC-tau-sigma}
\textstyle 
\sigma=\frac{1}{4R}\sqrt{\gamma(n\lambda+\delta\mu^2)}, \quad
\tau=\frac{1}{4R}\sqrt{\frac{\gamma}{n\lambda+\delta\mu^2}}, 
\end{equation}
and $\theta=\max\{\theta_x, \theta_y\}$ where
\begin{equation}\label{eqn:DF-SPDC-theta}
\textstyle
\theta_x = \left(1-\frac{\tau\sigma\delta\mu^2}{n(4   +2\sigma)}\right)
             \frac{1}{1+\tau\lambda}, \quad
\theta_y = \frac{1+((n-1)/n)\sigma/2}{1+\sigma/2},
\end{equation}
then we have
\begin{align*}
\textstyle
\left(\frac{1}{2\tau}+\frac{\lambda}{2}\right)\E\bigl[\|\xt-\xopt\|^2\bigr]
+\frac{\gamma}{4}\E\bigl[\Bdiv(\yopt,\yt)\bigr]  &\leq \theta^t C, \\
\textstyle
\E\left[\Lagr(\xt,\yopt) - \Lagr(\xopt,\yt)\right] &\leq \theta^t C,
\end{align*}
where
$
C = \left(\frac{1}{2\tau}\!+\!\frac{\lambda}{2}\right)\|\xini-\xopt\|^2
  +\left(\frac{1}{\sigma}\!+\!\frac{1}{2}\right) \Bdiv(\yopt,\yini).
$
\end{theorem}

Below we discuss the expected convergence rate
established in Theorem~\ref{thm:SPDC-convergence} in two special cases.

\vspace{-1ex}
\paragraph{The cases of $\sigma \mu^2 = 0$ but $\lambda > 0$.}
In this case we have
$\tau = \frac{1}{4R} \sqrt{ \frac{\gamma}{n \lambda } }$ and
$\sigma = \frac{1}{4R} \sqrt{ n \gamma\lambda }$, and
\begin{align*}
  \theta_x &= \textstyle \frac{1}{1+\tau\lambda} 
  = 1 - \frac{1}{1 + 4 R \sqrt{n/(\lambda \gamma)}}, \\
\theta_y &= \textstyle 
\frac{1+((n-1)/n)\sigma/2}{1+\sigma/2}
= 1 - \frac{1}{n + 8 R \sqrt{n/(\lambda \gamma)}}.
\end{align*}
These recover the convergence rate of
the standard SPDC algorithm \citep{ZhangXiao2015SPDC}.

\vspace{-1ex}
\paragraph{The cases of $\sigma \mu^2 > 0$ but $\lambda = 0$.}
In this case we have
\begin{align*}
\textstyle
\tau = \frac{1}{4R \mu} \sqrt{ \frac{\gamma}{ \delta } }, \quad  \sigma = \frac{\mu}{4R} \sqrt{ \delta \gamma },
\end{align*}
and
\vspace{-1ex}
\begin{align*}
\theta_x &= \textstyle
1 - \frac{\tau \sigma \delta \mu^2}{2 n(\sigma + 4  )} 
= 1 - \frac{\gamma \delta \mu^2}{32 n R^2} \cdot \frac{1}{ \sqrt{\gamma\delta}\mu/(4R) + 4 }, \\
\theta_y &= \textstyle 
\frac{1+((n-1)/n)\sigma/2}{1+\sigma/2}
= 1 - \frac{1}{n + 8nR/(\mu \sqrt{\gamma \delta})}.
\end{align*}
We note that the primal function now is $R^2/\gamma$-smooth
and $\delta \mu^2/n$-strongly convex. 
We discuss the following cases:
\vspace{-1ex}
\begin{itemize} \itemsep 0pt
    \item If $\sqrt{\gamma\delta}\mu > R$, 
      then we have $\theta_x \approx 1 - \frac{\sqrt{\gamma\delta}\mu}{8nR}$
      and $\theta_y\approx1-\frac{1}{n}$. 
      Therefore $\theta=\max\{\theta_x,\theta_y\}\approx 1-\frac{1}{n}$.
\item Otherwise, we have 
  $\theta_x \approx 1 - \frac{\gamma \delta \mu^2}{64 n R^2} $
  and $\theta_y$ is of the same order. 
  This is not an accelerated rate, and we have the same iteration complexity
 as SVRG.
\end{itemize}

Finally, we give concrete examples of how to compute the initial points
$\yini$ and $\vini$ such that $\viini = (\phi_i^*)'(\yiini)$.
\vspace{-1ex}
\begin{itemize}\itemsep 0pt
  \item For squared loss, $\phi_i(\alpha)=\frac{1}{2}(\alpha-b_i)^2$ and
    $\phi_i^*(\beta)=\frac{1}{2}\beta^2+b_i\beta$.
    So $\viini=(\phi_i^*)'(\yiini)=\yiini+b_i$.
  \item For logistic regression, we have $b_i\in\{1,-1\}$ and 
    $\phi_i(\alpha)=\log(1+e^{-b_i\alpha})$. The conjugate function is
      $\phi_i^*(\beta) 
      = (-b_i\beta)\log(-b_i\beta) + (1+b_i\beta)\log(1+b_i\beta)$
      if $b_i\beta\in[-1,0]$ and $+\infty$ otherwise.
      We can choose $\yiini\!=\!-\frac{1}{2}b_i$ and $\viini\!=\!0$
      such that $\viini \!=\! (\phi_i^*)'(\yiini)$.
\end{itemize}
For logistic regression, we have $\delta=0$ over the full domain of $\phi_i$.
However, each $\phi_i$ is locally strongly convex in
bounded domain \citep{bach2014adaptivity}: if $z \in [-B,B]$, then we know
$\delta = \min_{z} \phi_i{''}(z) \geq \exp(-B)/4$.
Therefore it is well suitable for an adaptation scheme similar to Algorithm~\ref{alg:BPD-Adapt-robust}
that do not require knowledge of either~$\delta$ or $\mu$.

\section{Preliminary experiments}

\begin{figure*}[t]
\psfrag{Number of passes}[tc][tc]{\footnotesize Number of passes}
\psfrag{Adaptive C-P}[cl][cl]{\scriptsize Ada-BPD}
\psfrag{Optimal C-P}[cl][cl]{\scriptsize Opt-BPD}
\psfrag{C-P}[cl][cl]{\scriptsize BPD}
\psfrag{AGD}[cl][cl]{\scriptsize Primal AG}
\begin{center}
\psfrag{Primal objective suboptimality}[bc][bc]{\footnotesize Primal optimality gap}
\includegraphics[width=0.33 \textwidth]{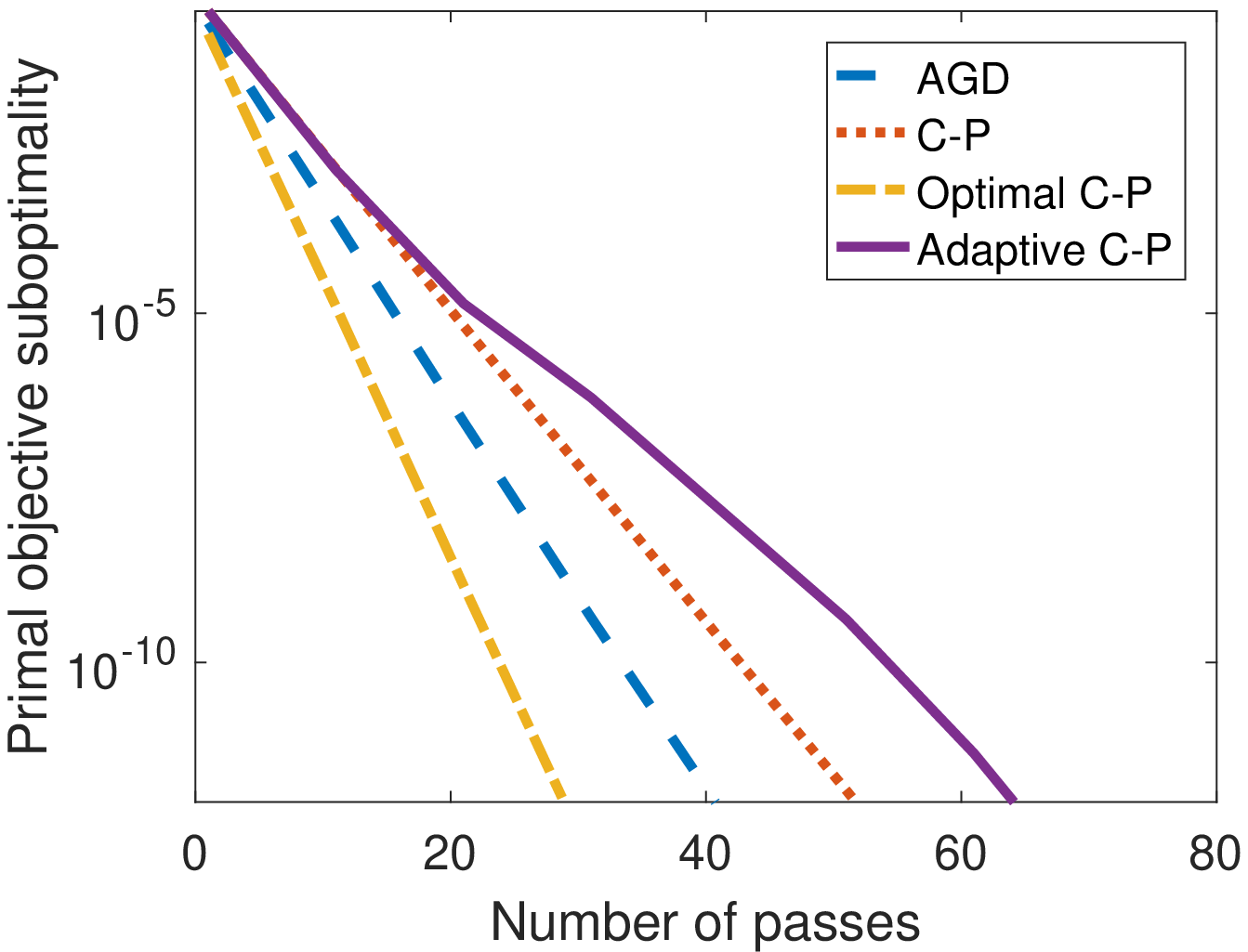}%
\psfrag{Primal objective suboptimality}[bc][bc]{}
\includegraphics[width=0.33 \textwidth]{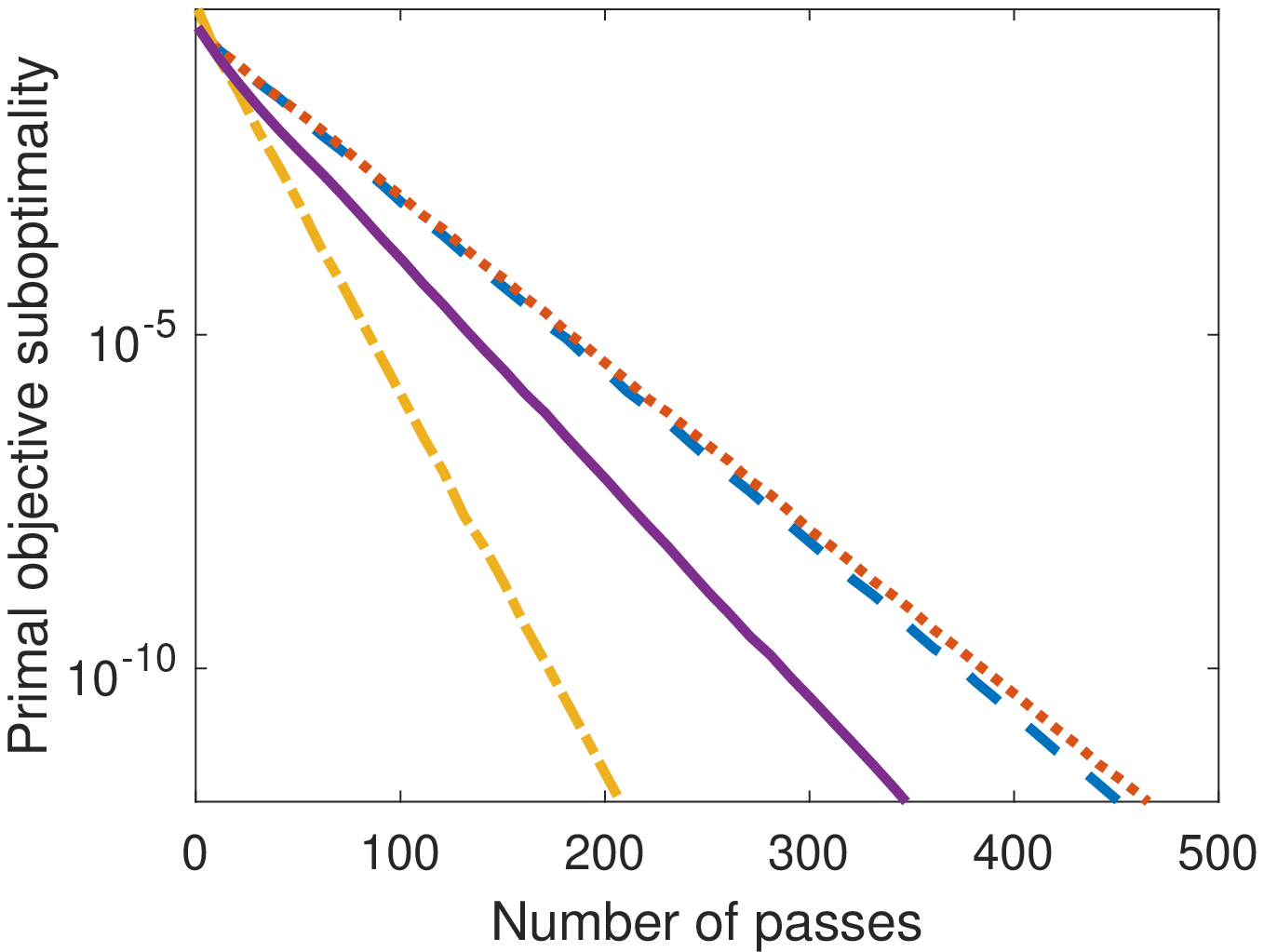}%
\includegraphics[width=0.33 \textwidth]{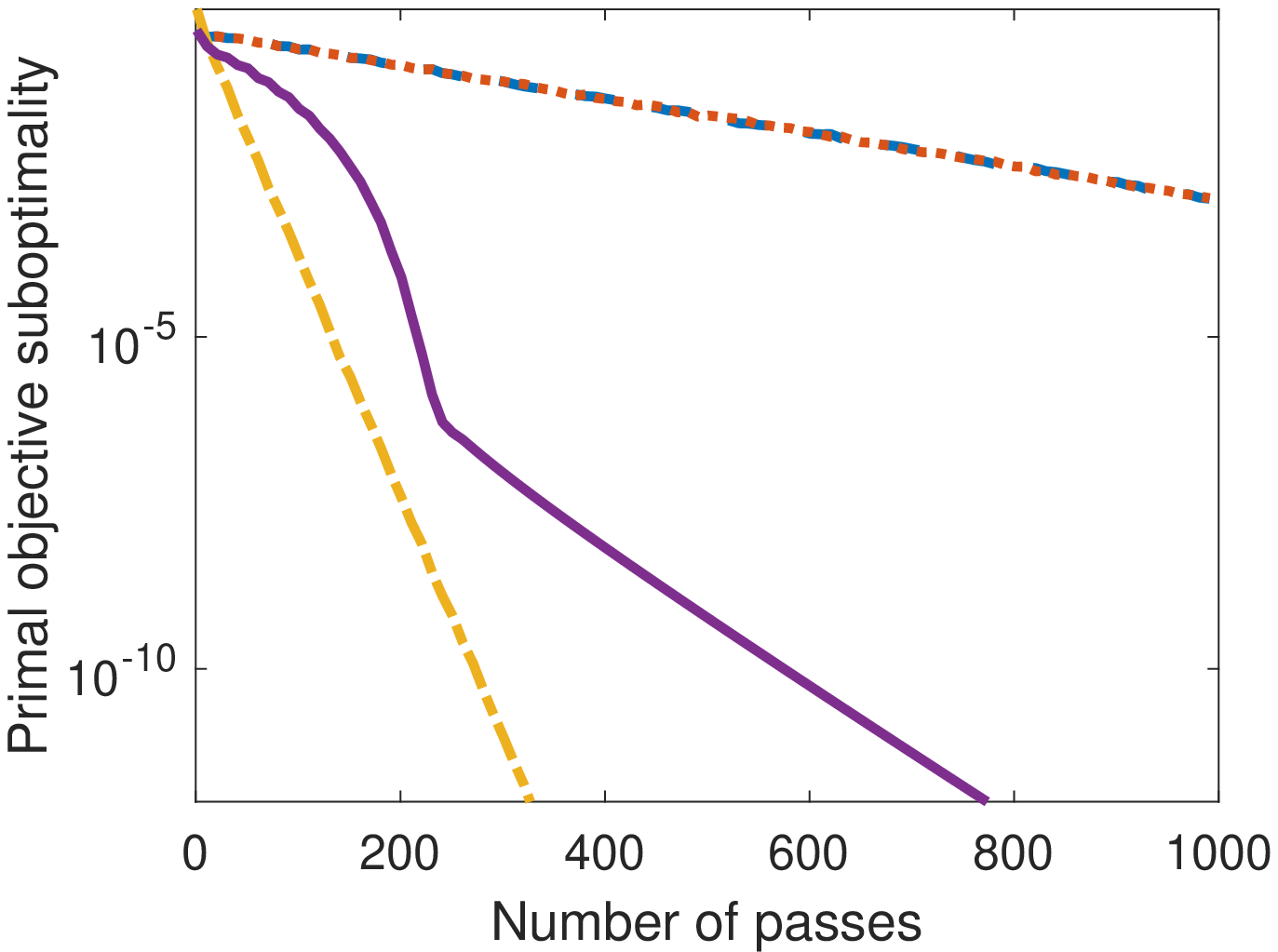}%
\end{center}
\makebox[0.34 \textwidth]{synthetic1, $\lambda = 1/n$}\makebox[0.34 \textwidth]{synthetic1, $\lambda = 10^{-2}/n$}\makebox[0.34 \textwidth]{synthetic1, $\lambda = 10^{-4}/n$}
\caption{Comparison of batch primal-dual algorithms for a
ridge regression problem with $n=5000$ and $d=3000$.}
\label{fig:batch}
\end{figure*}

We present preliminary experiments to demonstrate the effectiveness of our
proposed algorithms. 
First, we consider batch primal-dual algorithms for ridge regression
over a synthetic dataset.
The data matrix~$A$ has sizes $n=5000$ and $d=3000$, and its entries are
sampled from multivariate normal distribution with mean zero and covariance 
matrix $\Sigma_{ij} = 2^{|i-j|/2}$.
We normalize all datasets such that
$a_i = a_i/\left( \max_j \|a_j\| \right)$, 
to ensure the maximum norm of the data points is~$1$. 
We use $\ell_2$-regularization $g(x)=(\lambda/2)\|x\|^2$ with three
choices of parameter $\lambda$: $1/n$, $10^{-2}/n$ and $10^{-4}/n$, 
which represent the strong, medium, and weak levels of regularization, 
respectively.

Figure~\ref{fig:batch} shows the performance of four different algorithms:
the accelerated gradient algorithm for solving the primal minimization problem
(Primal AG) \cite{Nesterov2004book} using~$\lambda$ as strong convexity 
parameter, the BPD algorithm (Algorithm~\ref{alg:BPD}) that uses $\lambda$ as 
the strong convexity parameter (setting $\mu=0$), 
the optimal BPD algorithm (Opt-BPD) that uses $\mu=\sqrt{\lambdamin(A^T A)}$
explicitly computed from data,
and the Ada-BPD algorithm (Algorithm~\ref{alg:Ada-BPD}) with the robust 
adaptation heuristic (Algorithm~\ref{alg:BPD-Adapt-robust}) with
$T=10$, $\underline{c}=0.95$ and $\overline{c}=1.5$.
As expected, the performance of Primal-AG is very similar to BPD with the 
same strong convexity parameter.
The Opt-BPD fully exploits strong convexity from data, thus has the fastest
convergence. 
The Ada-BPD algorithm can partially exploit strong convexity from data
without knowledge of~$\mu$.

Next we compare the DF-SPDC (Algorithm~\ref{alg:Ada-SPDC} without adaption) and 
ADF-SPDC (Algorithm~\ref{alg:ADF-SPDC} with adaption) 
against several state-of-the-art randomized algorithms for ERM:
SVRG \citep{johnson2013accelerating}, 
SAGA \citep{defazio2014saga} 
Katyusha \citep{allen2016katyusha} 
and the standard SPDC method \citep{ZhangXiao2015SPDC}. 
For SVRG and Katyusha (an accelerated variant of SVRG), we choose the variance
reduction period as $m = 2n$.
The step sizes of all algorithms are set as their
original paper suggested. 
For Ada-SPDC and ADF-SPDC, we use the robust adaptation scheme with
$T=10$, $\underline{c}=0.95$ and $\overline{c}=1.5$.

We first compare these randomized algorithms for ridge regression over 
the same synthetic data described above and 
the \texttt{cpuact} data from
the LibSVM website\footnote{\url{https://www.csie.ntu.edu.tw/~cjlin/libsvm/}}. 
The results are shown in Figure~\ref{fig:ridge}.
With relatively strong regularization $\lambda=1/n$, 
all methods perform similarly as predicted by theory.
For the synthetic dataset 
With $\lambda=10^{-2}/n$, 
the regularization is weaker but still stronger than the hidden strong 
convexity from data, so the accelerated algorithms (all variants of SPDC and
Katyusha) perform better than SVRG and SAGA.
With $\lambda=10^{-4}/n$, it looks that the strong convexity from data dominates
the regularization. 
Since the non-accelerated algorithms (SVRG and SAGA) may automatically exploit
strong convexity from data, they become 
faster than the non-adaptive accelerated methods (Katyusha, SPDC and DF-SPDC).
The adaptive accelerated method, ADF-SPDC, has the fastest convergence.
This shows that our theoretical results (which predict no acceleration in
this case) can be further improved.

Finally we compare these randomized algorithm for logistic regression on
the \texttt{rcv1} dataset (from LibSVM website) and
another synthetic dataset with $n=5000$ and $d=500$, generated similarly
as before but with covariance matrix $\Sigma_{ij} = 2^{|i-j|/100}$.
For the standard SPDC, we solve the dual proximal mapping using a few steps
of Newton's method to high precision.
The dual-free SPDC algorithms only use gradients of the logistic function.
The results are presented in Figure~\ref{fig:logistic}.
for both datasets, the strong convexity from data is very weak (or none),
so the accelerated algorithms performs better. 

\begin{figure*}[t]
\psfrag{Number of passes}[tc][tc]{\footnotesize Number of passes}
\psfrag{SPDC-BD-Adaptive}[cl][cl]{\scriptsize ADF-SPDC}
\psfrag{SPDC-BD}[cl][cl]{\scriptsize DF-SPDC}
\psfrag{SPDC}[cl][cl]{\scriptsize SPDC}
\psfrag{SVRG}[cl][cl]{\scriptsize SVRG}
\psfrag{SAGA}[cl][cl]{\scriptsize SAGA}
\psfrag{Katyusha}[cl][cl]{\scriptsize Katyusha}
\begin{center}
\psfrag{Primal objective suboptimality}[bc][bc]{\footnotesize Primal optimality gap}
\includegraphics[width=0.33 \textwidth]{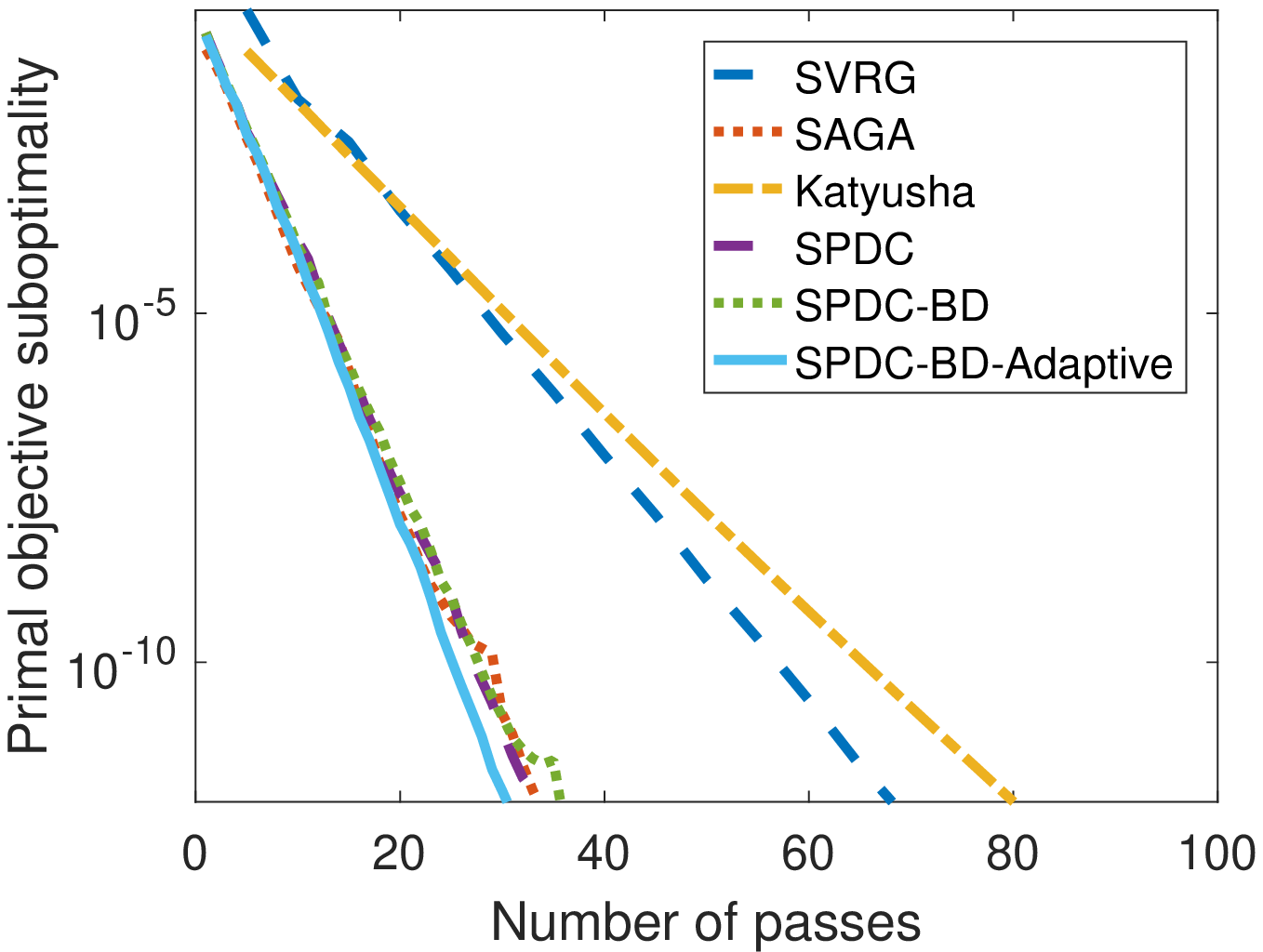}%
\psfrag{Primal objective suboptimality}[bc][bc]{}
\includegraphics[width=0.33 \textwidth]{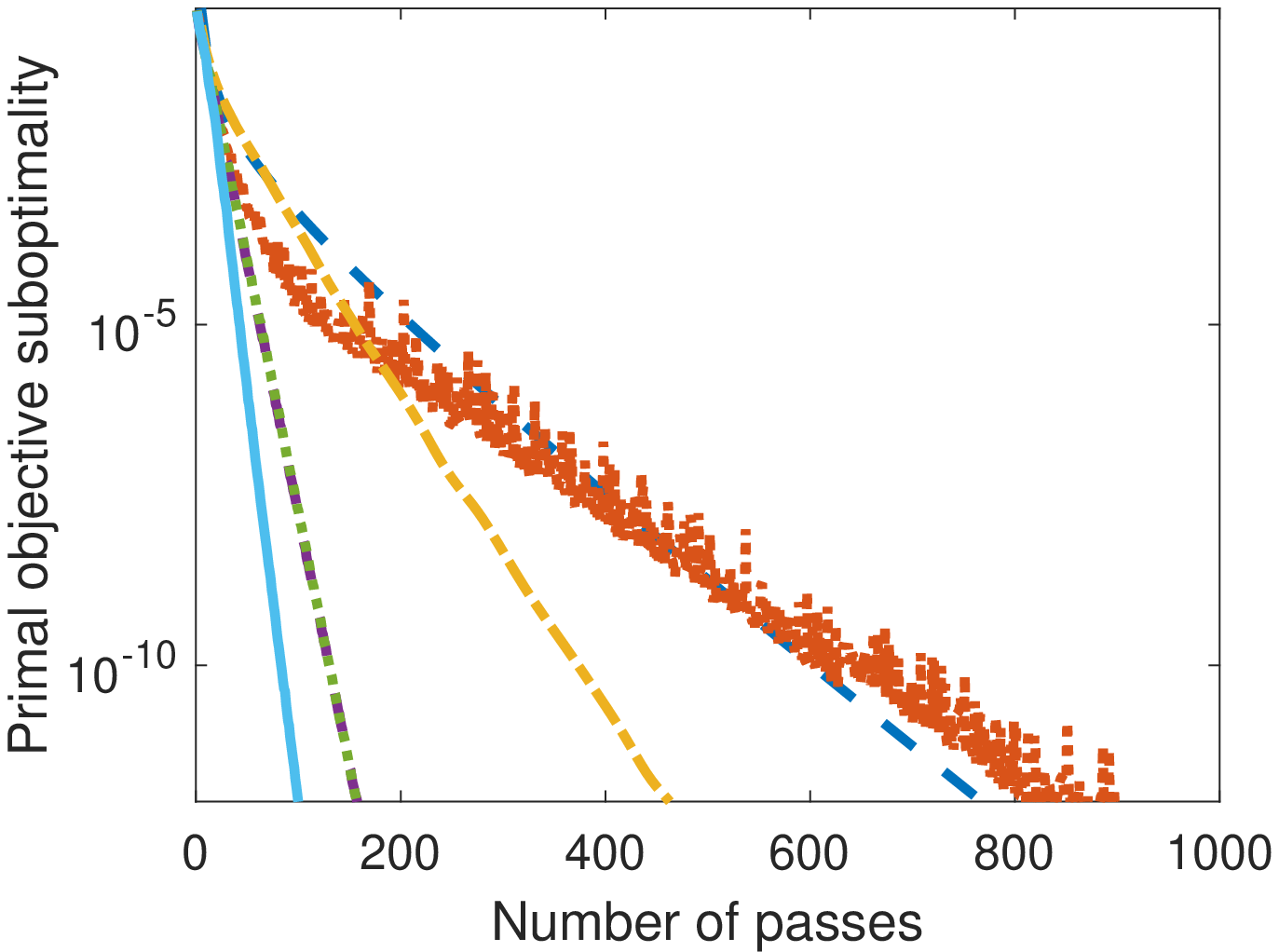}%
\includegraphics[width=0.33 \textwidth]{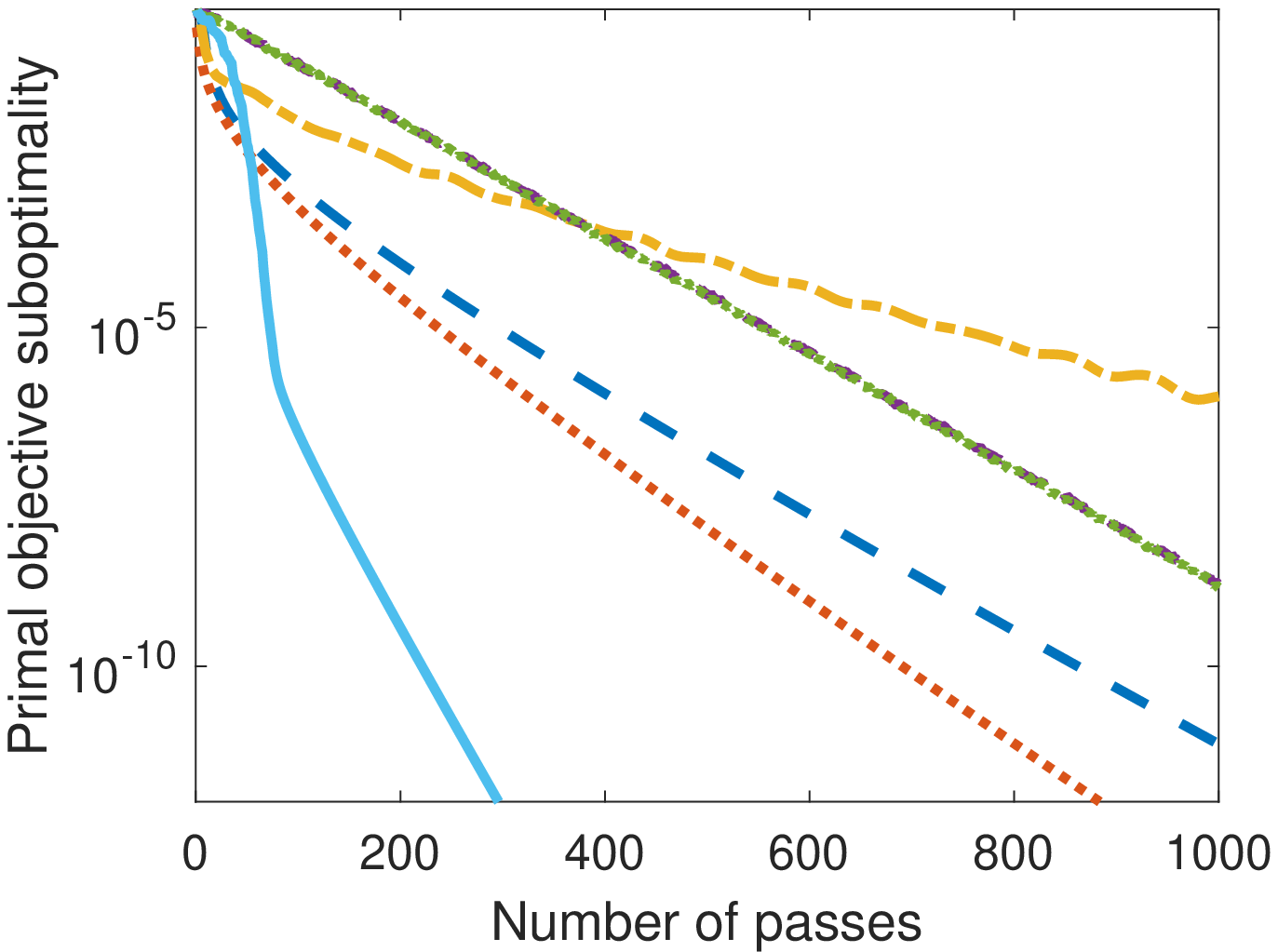}%
\end{center}
\makebox[0.34 \textwidth]{synthetic1, $\lambda = 1/n$}\makebox[0.34 \textwidth]{synthetic1, $\lambda = 10^{-2}/n$}\makebox[0.34 \textwidth]{synthetic1, $\lambda = 10^{-4}/n$}
\begin{center}
\psfrag{Primal objective suboptimality}[bc][bc]{\footnotesize Primal optimality gap}
\includegraphics[width=0.33 \textwidth]{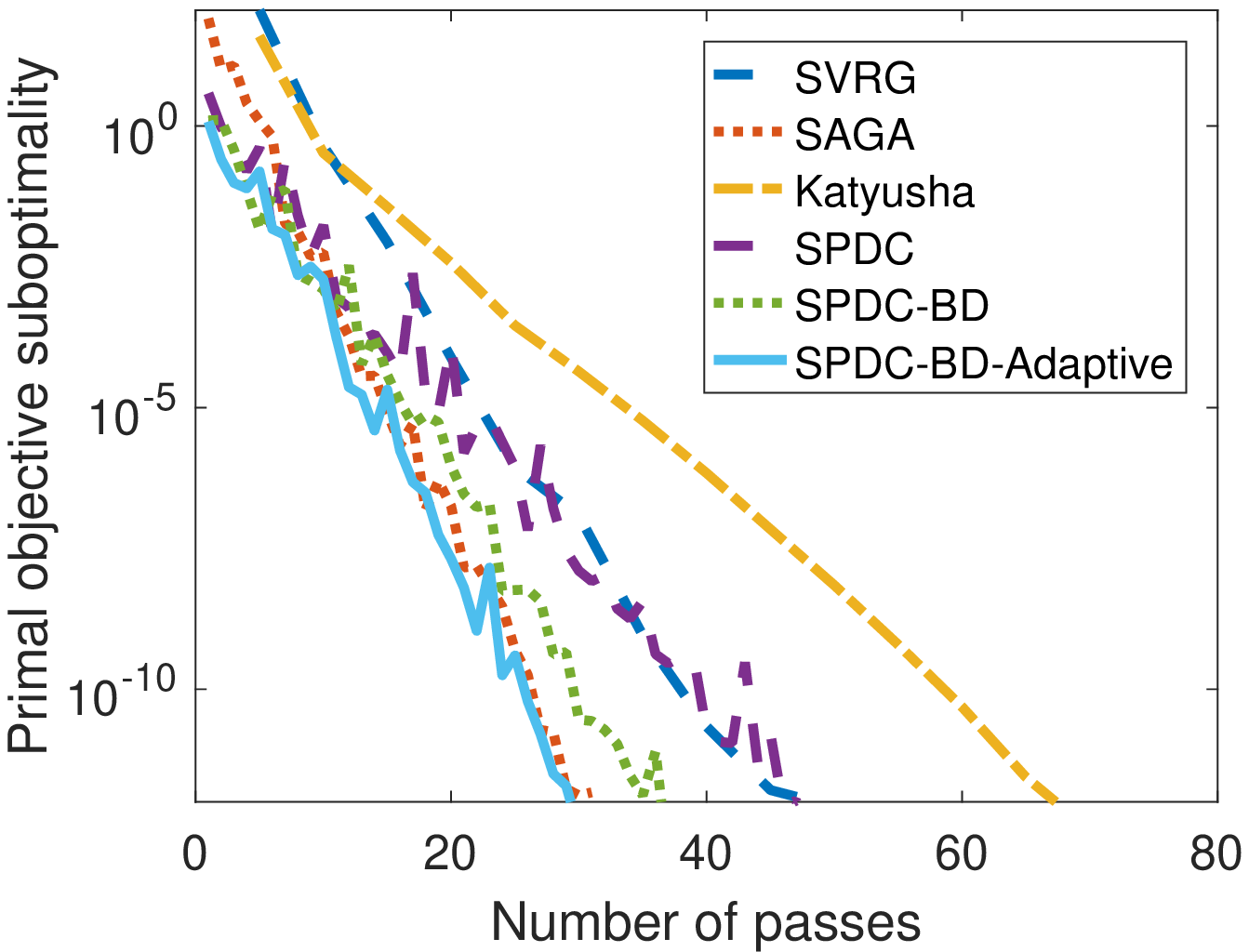}%
\psfrag{Primal objective suboptimality}[bc][bc]{}
\includegraphics[width=0.33 \textwidth]{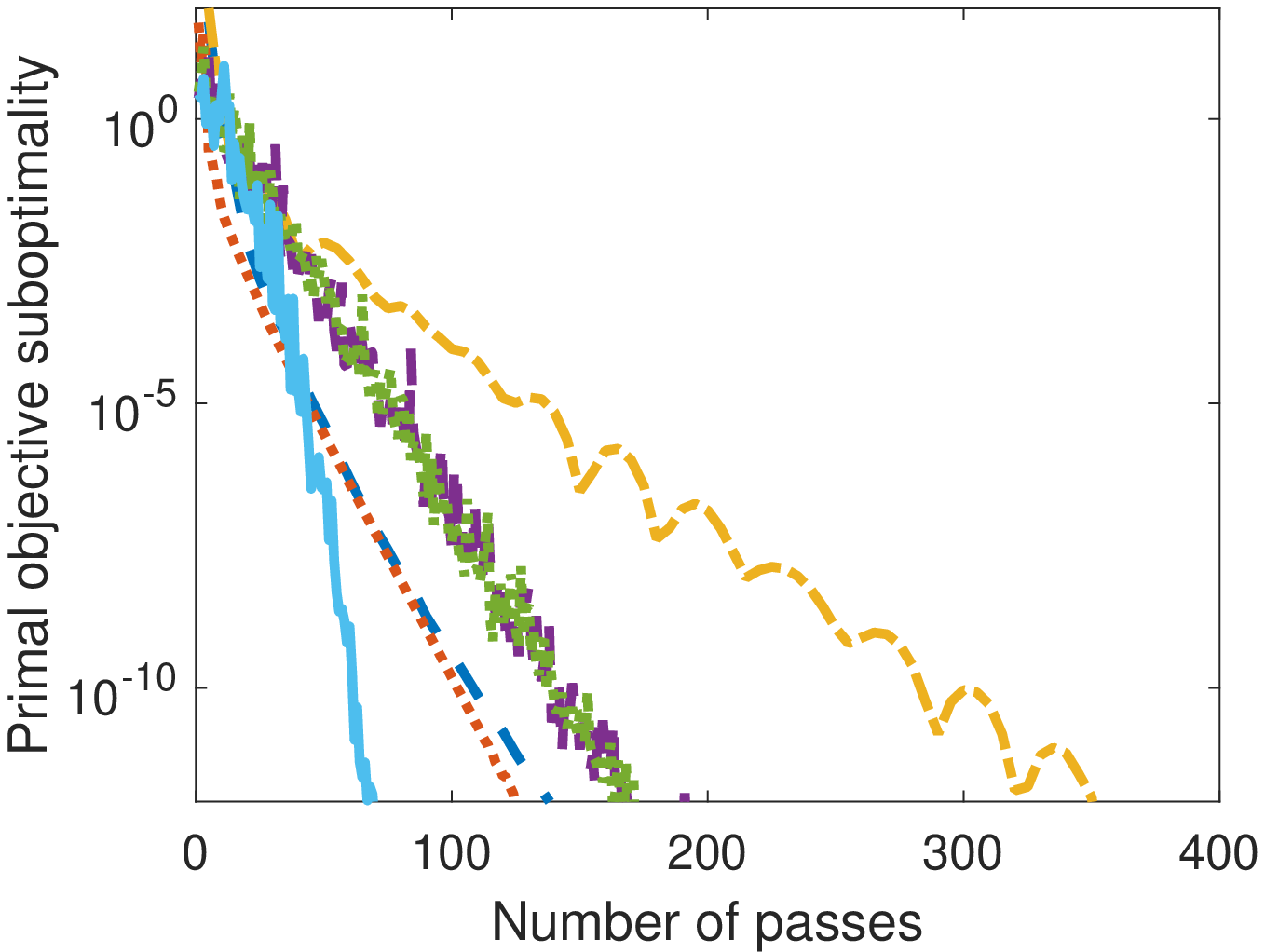}%
\includegraphics[width=0.33 \textwidth]{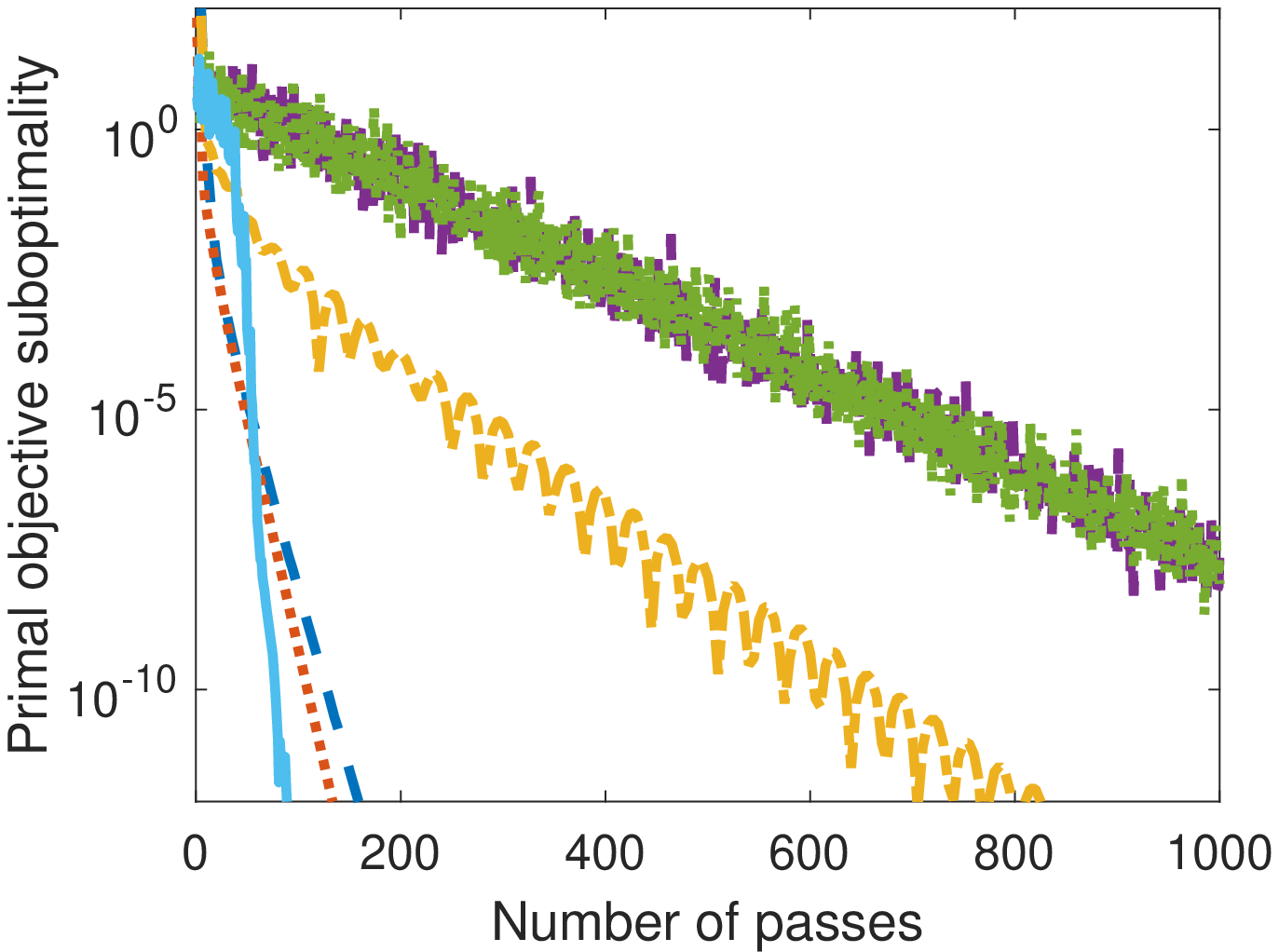}%
\end{center}
\makebox[0.34 \textwidth]{cpuact, $\lambda = 1/n$}\makebox[0.34 \textwidth]{cpuact, $\lambda = 10^{-2}/n$}\makebox[0.34 \textwidth]{cpuact, $\lambda = 10^{-4}/n$}
\caption{Comparison of randomized algorithms for ridge regression problems.}
\label{fig:ridge}
\end{figure*}

\begin{figure*}[t]
\psfrag{Number of passes}[tc][tc]{\footnotesize Number of passes}
\psfrag{SPDC-BD-Adaptive}[cl][cl]{\scriptsize ADF-SPDC}
\psfrag{SPDC-BD}[cl][cl]{\scriptsize DF-SPDC}
\psfrag{SPDC}[cl][cl]{\scriptsize SPDC}
\psfrag{SVRG}[cl][cl]{\scriptsize SVRG}
\psfrag{SAGA}[cl][cl]{\scriptsize SAGA}
\psfrag{Katyusha}[cl][cl]{\scriptsize Katyusha}
\begin{center}
\psfrag{Primal objective suboptimality}[bc][bc]{\footnotesize Primal optimality gap}
\includegraphics[width=0.33 \textwidth]{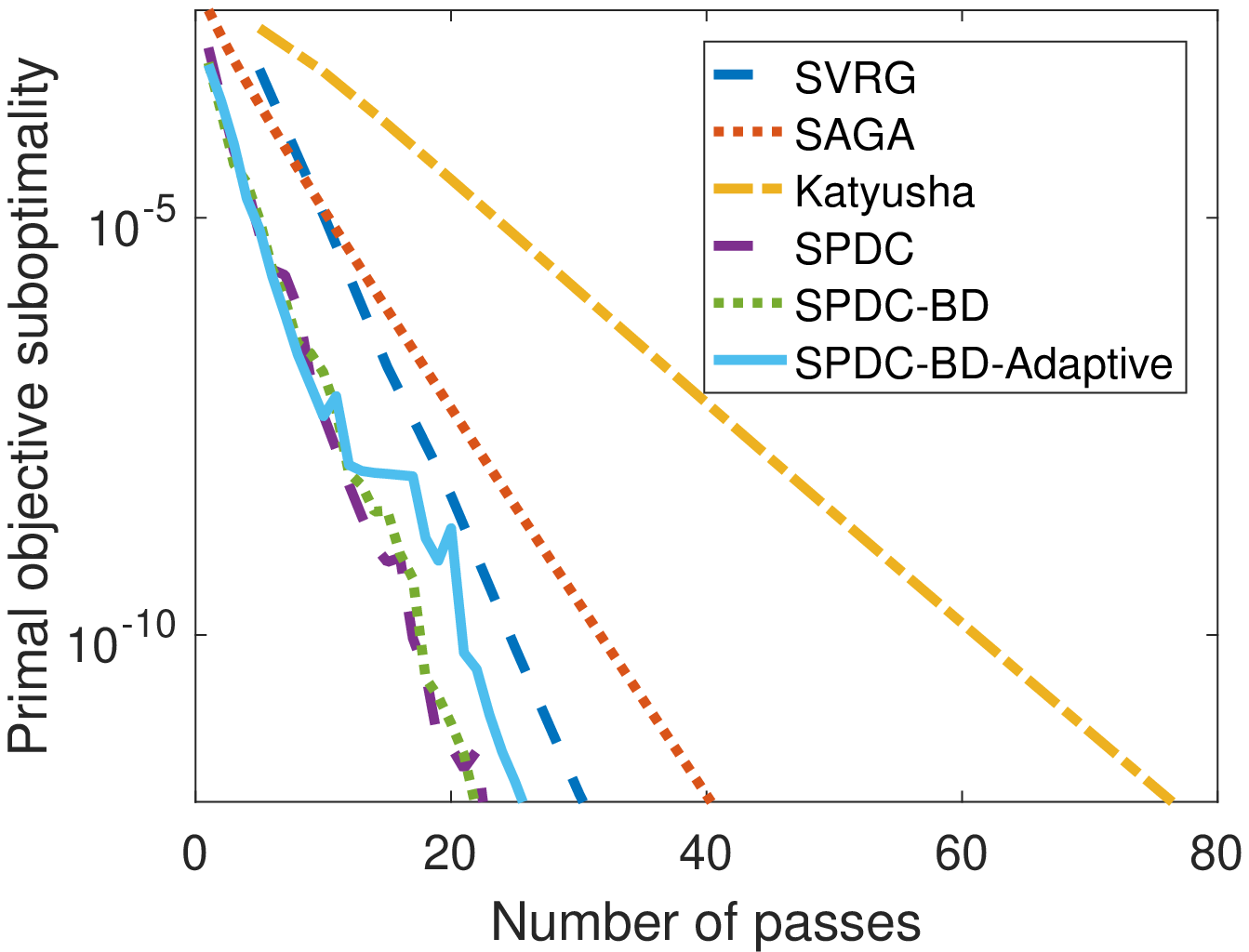}%
\psfrag{Primal objective suboptimality}[bc][bc]{}
\includegraphics[width=0.33 \textwidth]{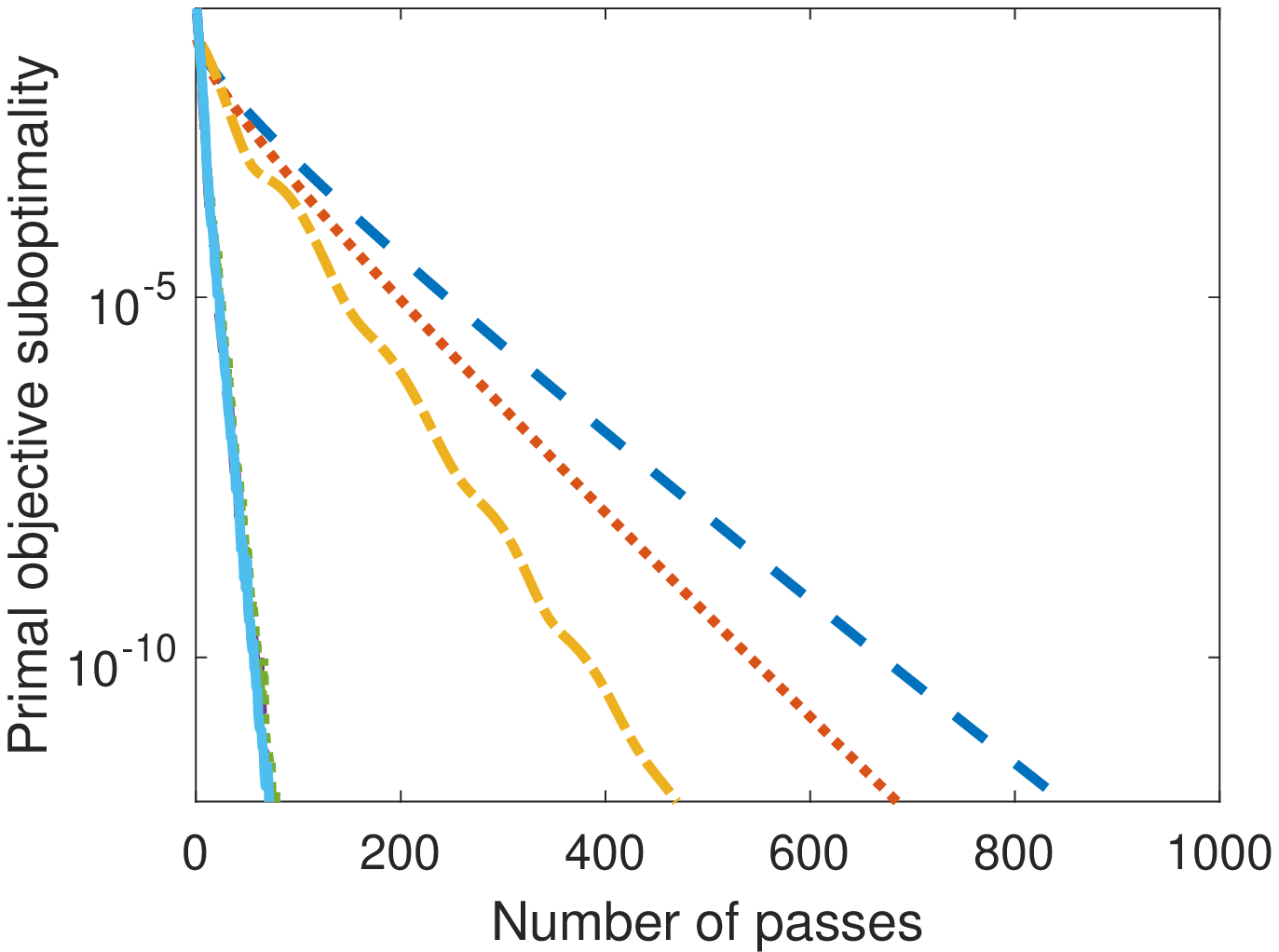}%
\includegraphics[width=0.33 \textwidth]{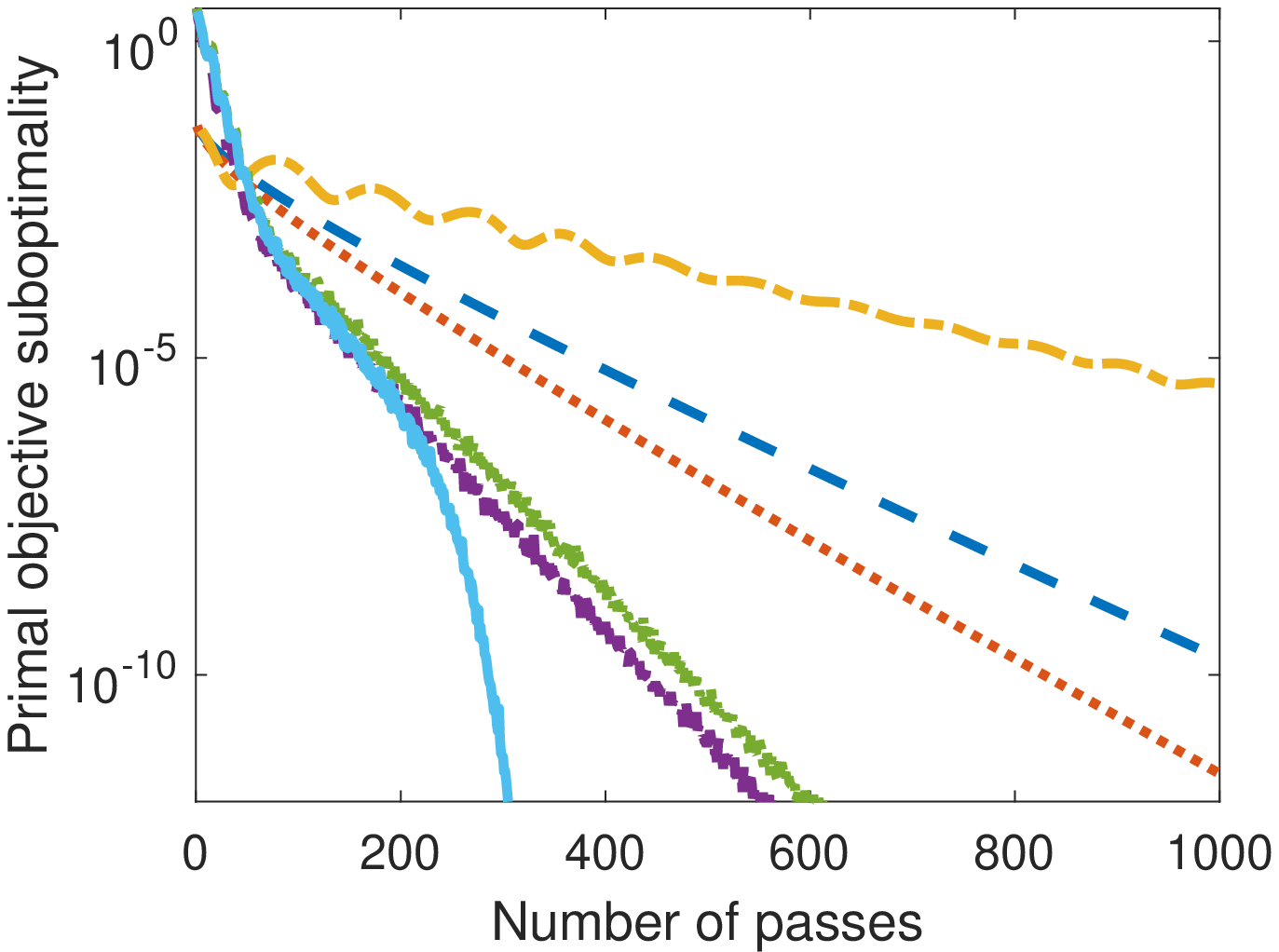}%
\end{center}
\makebox[0.34 \textwidth]{synthetic2, $\lambda = 1/n$}\makebox[0.34 \textwidth]{synthetic2, $\lambda = 10^{-2}/n$}\makebox[0.34 \textwidth]{synthetic2, $\lambda = 10^{-4}/n$}
\begin{center}
\psfrag{Primal objective suboptimality}[bc][bc]{\footnotesize Primal optimality gap}
\includegraphics[width=0.33 \textwidth]{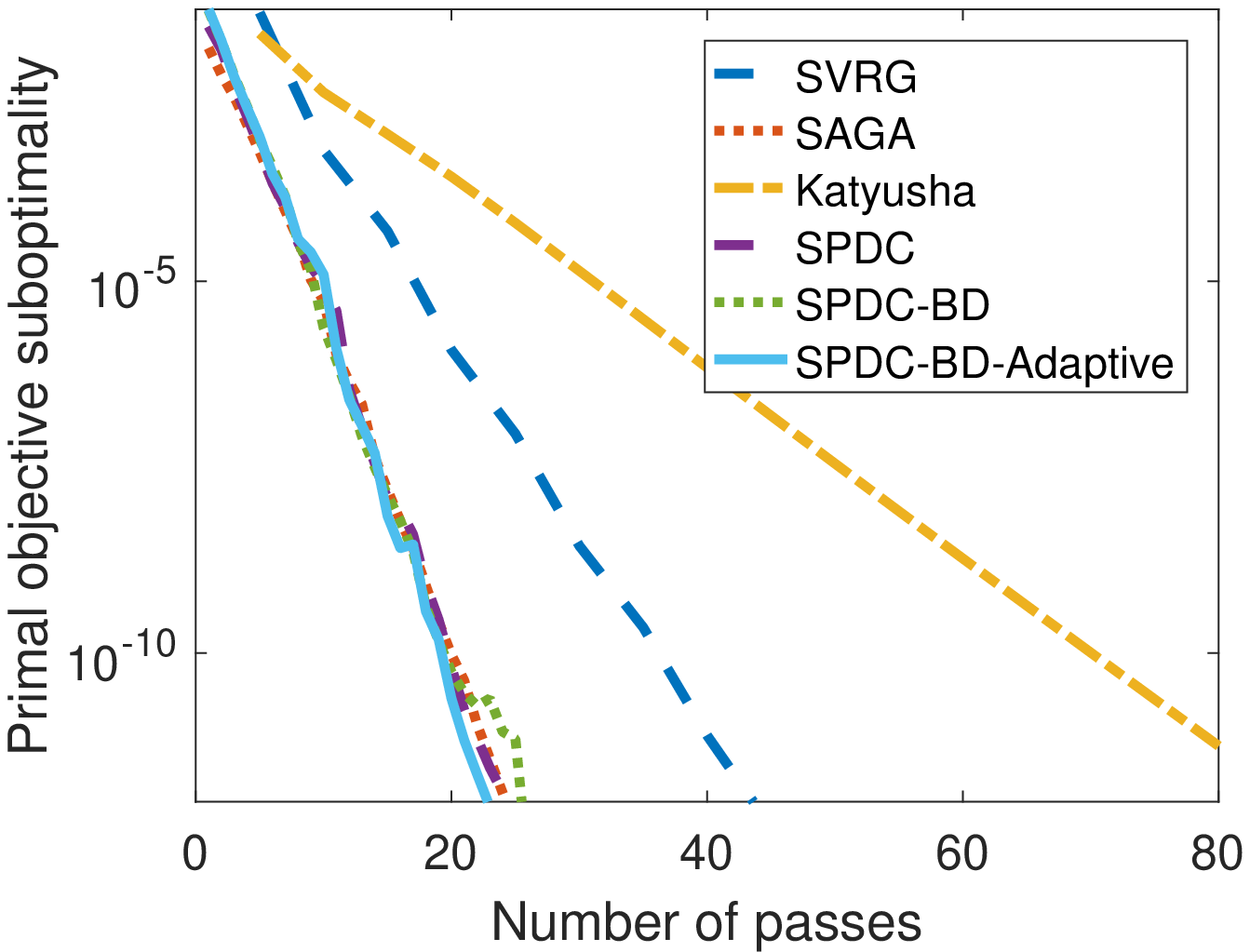}%
\psfrag{Primal objective suboptimality}[bc][bc]{}
\includegraphics[width=0.33 \textwidth]{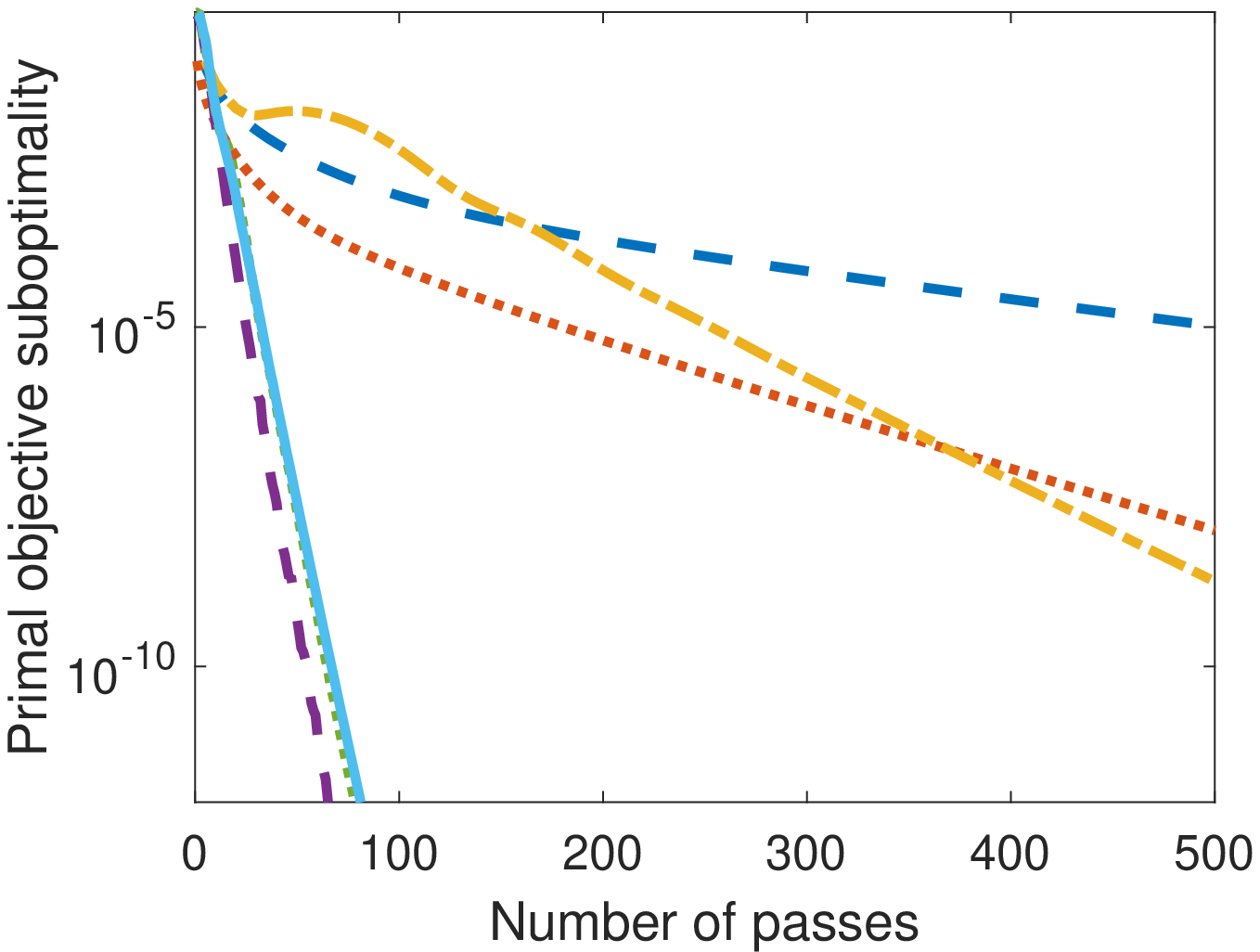}%
\includegraphics[width=0.33 \textwidth]{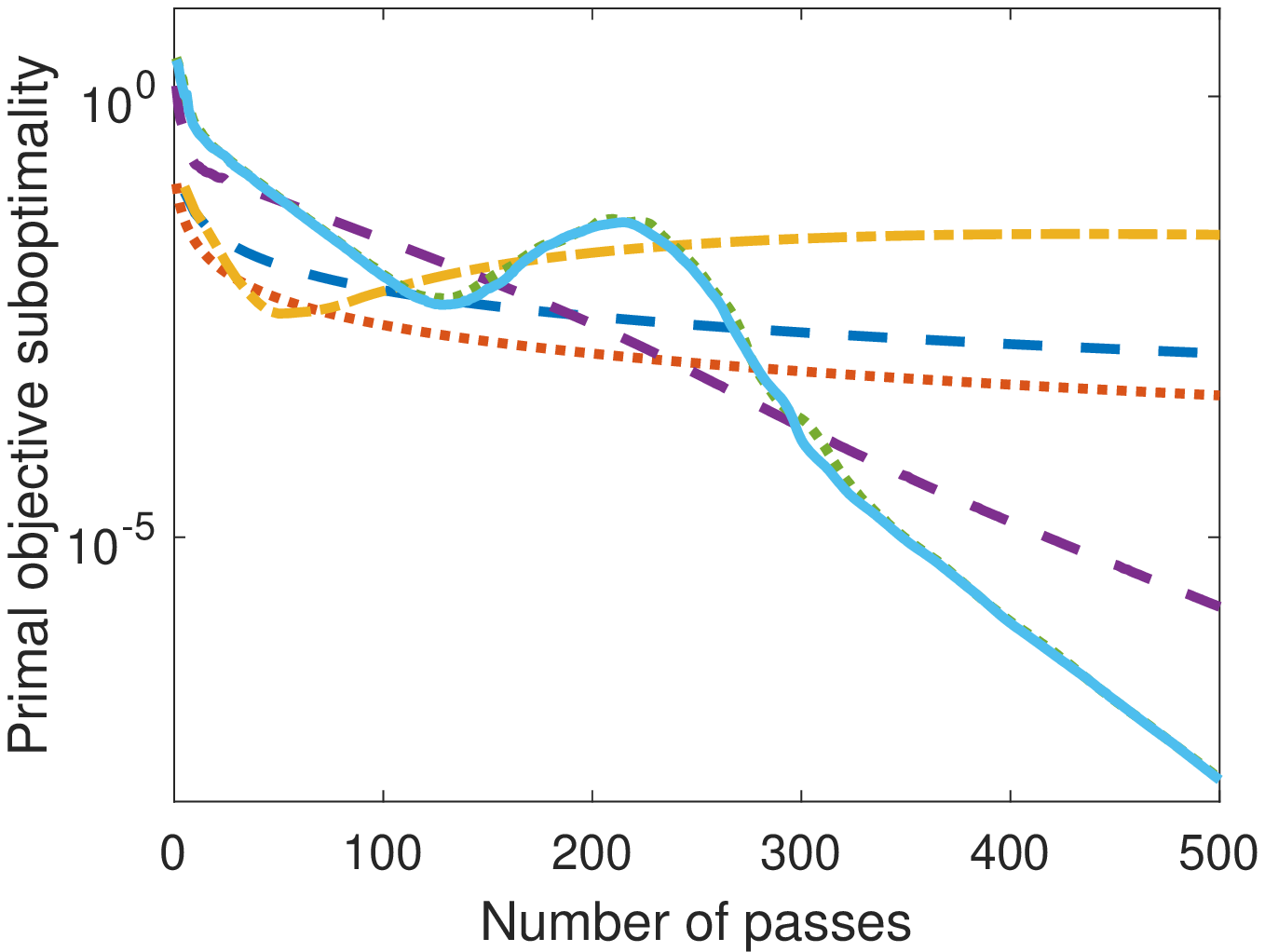}%
\end{center}
\makebox[0.34 \textwidth]{rcv1, $\lambda = 1/n$}\makebox[0.34 \textwidth]{rcv1, $\lambda = 10^{-2}/n$}\makebox[0.34 \textwidth]{rcv1, $\lambda = 10^{-4}/n$}
\caption{Comparison of randomized algorithms for logistic regression problems.}
\label{fig:logistic}
\end{figure*}

\section{Conclusions}
We have shown that primal-dual first-order algorithms are capable of
exploiting strong convexity from data, if the algorithmic parameters are
chosen appropriately. While they may depends on problem dependent constants
that are unknown, we developed heuristics for adapting the parameters 
on the fly and obtained improved performance in experiments.
It looks that our theoretical characterization of the convergence rates
can be further improved, as  our experiments often demonstrate significant
acceleration in cases where our theory does not predict acceleration.

\clearpage
\balance
\bibliography{PDsaddle}
\bibliographystyle{icml2017}

\onecolumn
\appendix

In the following appendices,
we provide detailed proofs of theorems stated in the main paper. 
In Section~\ref{sec:fundamental_ineq} 
we first prove a basic inequality which is useful
throughout the rest of the convergence analysis. 
Section~\ref{sec:convergence_general} contains general analysis of the 
batch primal-dual algorithm that are common for proving both
Theorem~\ref{thm:BPD-convergence} and Theorem~\ref{thm:DF-BPD-convergence}.
Sections~\ref{sec:analysis-Euclidean}, \ref{sec:proof_dfbpd}, 
\ref{sec:proof_spdc} and~\ref{sec:proof_adf_spdc} give proofs for
Theorem~\ref{thm:BPD-convergence}, Theorem~\ref{thm:DF-BPD-convergence}, 
Theorem~\ref{thm:SPDC-convergence} and Theorem~\ref{thm:DF-SPDC-convergence},
respectively.

\section{A basic lemma}
\label{sec:fundamental_ineq}

\begin{lemma} \label{lem:sc-min-3point}
  Let $h$ be a strictly convex function and $\Bdiv_h$ be its Bregman divergence.
  Suppose $\psi$ is $\nu$-strongly convex with respect to $\Bdiv_h$ 
  and $1/\delta$-smooth (with respect to the Euclidean norm), and 
  \[
    \yhat = \argmin_{y\in C}~\bigl\{\psi(y) + \eta\Bdiv_h(y,\ybar)\bigr\},
  \]
  where $C$ is a compact convex set that lies within the relative interior
  of the domains of $h$ and $\psi$ 
  (i.e., both $h$ and $\psi$ are differentiable over~$C$). 
  Then for any $y\in C$ and $\rho\in[0,1]$, we have
  \[
    \psi(y)+\eta\Bdiv_h(y,\xbar) \geq \psi(\yhat)+\eta\Bdiv_h(\yhat,\ybar) 
    +\bigl(\eta+(1-\rho)\nu\bigr)\Bdiv_h(y,\yhat)
    +\frac{\rho\delta}{2}\left\|\nabla \psi(y)-\nabla \psi(\yhat)\right\|^2 .
  \]
\end{lemma}

\begin{proof}
The minimizer $\yhat$ satisfies the following first-order optimality condition:
\[
  \left\langle \nabla \psi(\yhat)+\eta\nabla\Bdiv_h(\yhat,\ybar),\; y-\yhat
  \right\rangle \geq 0, \quad \forall\; y\in C.
\]
Here $\nabla \Bdiv$ denotes partial gradient of the Bregman divergence 
with respect to its first argument, i.e.,
$\nabla\Bdiv(\yhat,\ybar) = \nabla h(\yhat) - \nabla h(\ybar)$.
So the above optimality condition is the same as
\begin{equation}\label{eqn:sc-min-opt-C}
  \bigl\langle \nabla \psi(\yhat)+\eta(\nabla h(\yhat)-\nabla h(\ybar)),\;
  y-\yhat \bigr\rangle \geq 0, \quad \forall\; y\in C.
\end{equation}
Since $\psi$ is $\nu$-strongly convex with respect to $\Bdiv_h$
and $1/\delta$-smooth, we have 
\begin{align*}
\psi(y) & \geq \psi(\yhat) + \langle \nabla\psi(\yhat),y-\xhat \rangle 
+ \nu \Bdiv_h(y,\yhat), \\
\psi(y) & \geq \psi(\yhat) + \langle \nabla\psi(\yhat),y-\yhat \rangle 
+ \frac{\delta}{2} \bigl\|\nabla\psi(y)-\nabla\psi(\yhat)\bigr\|^2.
\end{align*}
For the second inequality, see, e.g., Theorem~2.1.5 in \citet{Nesterov2004book}.
Multiplying the two inequalities above by $(1-\rho)$ and $\rho$ respectively
and adding them together, we have
\begin{align*}
\psi(y) & \geq \psi(\yhat) + \langle \nabla\psi(\yhat),y-\yhat \rangle 
+ (1-\rho)\nu \Bdiv_h(y,\yhat) 
+ \frac{\rho\delta}{2} \bigl\|\nabla\psi(y)-\nabla\psi(\yhat)\bigr\|^2.
\end{align*}
The Bregman divergence $\Bdiv_h$ satisfies the following equality:
\[
  \Bdiv_h(y,\ybar) = \Bdiv_h(y,\yhat) + \Bdiv_h(\yhat,\ybar)
  +\bigl\langle \nabla h(\yhat)-\nabla h(\ybar),\; y-\yhat \bigr\rangle.
\]
We multiply this equality by~$\eta$ and add it to the last inequality to obtain
\begin{eqnarray*}
  \psi(y) + \eta\Bdiv_h(y,\ybar)
&\geq& \psi(\yhat) + \eta \Bdiv_h(y,\yhat) + \bigl(\eta+(1-\rho)\nu\bigr) 
       \Bdiv_h(\yhat,\ybar) 
     + \frac{\rho\delta}{2} \bigl\|\nabla\psi(y)-\nabla\psi(\yhat)\bigr\|^2 \\
&& +\bigl\langle \nabla \psi(\yhat)+\eta(\nabla h(\yhat)-\nabla h(\ybar)),\; 
  y-\yhat \bigr\rangle.
\end{eqnarray*}
Using the optimality condition in~\eqref{eqn:sc-min-opt-C}, the last term
of inner product is nonnegative and thus can be dropped, 
which gives the desired inequality.
\end{proof}

\section{Common Analysis of Batch Primal-Dual Algorithms}
\label{sec:convergence_general}

We consider the general primal-dual update rule as:
\begin{center}
\begin{boxedminipage}{0.99\textwidth}
\vspace{0.2em}
\textbf{Iteration:} 
$(\xhat,\yhat) = \mathrm{PD}_{\tau,\sigma}(\xbar,\ybar,\xtld,\ytld)$
\vspace{-1ex}
\begin{eqnarray}
  \xhat &=& \argmin_{x\in\reals^d} \left\{ g(x) + \ytld^T Ax + \frac{1}{2\tau} \|x-\xbar\|^2 \right\},  \label{eqn:PD-x-update} \\
  \yhat &=& \argmin_{y\in\reals^n} \left\{ f^*(y) - y^T A\xtld + \frac{1}{\sigma} \Bdiv(y,\ybar) \right\}.  \label{eqn:PD-y-update} 
\end{eqnarray}
\end{boxedminipage}
\end{center}

Each iteration of Algorithm~\ref{alg:BPD} is equivalent to the following 
specification of $\mathrm{PD}_{\tau,\sigma}$:
\begin{eqnarray}
 && \xhat = \xtp, \qquad \xbar = \xt, \qquad \xtld=\xt+\theta(\xt-\xtm), 
 \nonumber\\ 
 && \yhat = \ytp, \qquad \ybar = \yt, \qquad \ytld=\ytp .
 \label{eqn:y-first-subs}
\end{eqnarray}

Besides Assumption~\ref{asmp:batch-PD}, we also assume that
$f^*$ is $\nu$-strongly convex with respect to a kernel function $h$, 
    i.e., 
    \[
      f^*(y') - f^*(y) - \langle\nabla f^*(y),y'-y\rangle \geq \nu\Bdiv_h(y',y),
    \]
    where 
    $\Bdiv_h$ is the Bregman divergence defined as
    \[
      \Bdiv_h(y',y) = h(y') - h(y) - \langle \nabla h(y), y'-y\rangle. 
    \]
We assume that $h$ is $\gamma'$-strongly convex and $1/\delta'$-smooth.
Depending on the kernel function~$h$, this assumption on $f^*$ may
impose additional restrictions on~$f$. In this paper, we are 
mostly interested in two special cases: $h(y)=(1/2)\|y\|^2$ 
and $h(y)=f^*(y)$ (for the latter we always have $\nu=1$).
From now on, we will omit the subscript~$h$ and use $\Bdiv$ denote the Bregman divergence.

Under the above assumptions, any solution $(\xopt, \yopt)$ to the 
saddle-point problem~\eqref{eqn:batch-saddle} satisfies the 
optimality condition:
\begin{eqnarray} 
  -A^T y^\star &\in& \partial g(x^\star), \label{eqn:saddle-x-optimality} \\
  A x^\star &=& \nabla f^*(y^\star). \label{eqn:saddle-y-optimality}
\end{eqnarray}
The optimality conditions for the updates described in 
equations~\eqref{eqn:PD-x-update} and~\eqref{eqn:PD-y-update} are
\begin{eqnarray}
  -A^T \ytld + \frac{1}{\tau}(\xbar-\xhat) &\in& \partial g(\xhat),
  \label{eqn:PD-x-optimality} \\
  A \xtld - \frac{1}{\sigma}\bigl(\nabla h(\yhat)-\nabla h(\ybar)\bigr) 
  &=& \nabla f^*(\yhat) .
  \label{eqn:PD-y-optimality}
\end{eqnarray}

Applying Lemma~\ref{lem:sc-min-3point} to the dual minimization step
in~\eqref{eqn:PD-y-update} with $\psi(y)=f^*(y)-y^T A\xtld$, $\eta=1/\sigma$
$y=\yopt$ and $\rho=1/2$, we obtain
\begin{eqnarray}\label{eqn:sc-min-dual}
  f^*(\yopt)-\yopt^T A\xtld+\frac{1}{\sigma}\Bdiv(\yopt,\ybar)
&\geq& f^*(\yhat)-\yhat^T A\xtld+\frac{1}{\sigma}\Bdiv(\yhat,\ybar) \nonumber\\
&& +\Bigl(\frac{1}{\sigma}+\frac{\nu}{2}\Bigr) \Bdiv(\yopt,\yhat)
  + \frac{\delta}{4} \bigl\|\nabla f^*(\yopt)-\nabla f^*(\yhat)\bigr\|^2.
\end{eqnarray}
Similarly, for the primal minimization step in~\eqref{eqn:PD-x-update}, 
we have (setting $\rho=0$)
\begin{equation}\label{eqn:sc-min-primal}
  g(\xopt)+\ytld^T A\xopt+\frac{1}{2\tau}\|\xopt-\xbar\|^2
  \geq
  g(\xhat)+\ytld^T A\xhat+\frac{1}{2\tau}\|\xhat-\xbar\|^2
  + \frac{1}{2}\Bigl(\frac{1}{\tau}+\lambda\Bigr) \|\xopt-\xhat\|^2.
\end{equation}
Combining the two inequalities above with the definition 
$\Lagr(x,y)=g(x)+y^T Ax -f^*(y)$,
we get
\begin{eqnarray*}
\Lagr(\xhat,\yopt) - \Lagr(\xopt,\yhat)
&=& g(\xhat)+\yopt^T A\xhat-f^*(\yopt)-g(\xopt)-\yhat^T A\xopt + f^*(\yhat)\\
&\leq& \frac{1}{2\tau}\|\xopt-\xbar\|^2 + \frac{1}{\sigma}\Bdiv(\yopt,\ybar)
  - \frac{1}{2}\Bigl(\frac{1}{\tau}+\lambda\Bigr) \|\xopt-\xhat\|^2
  - \Bigl(\frac{1}{\sigma}+\frac{\nu}{2}\Bigr) \Bdiv(\yopt,\yhat) \\
&& -\,\frac{1}{2\tau}\|\xhat-\xbar\|^2 - \frac{1}{\sigma}\Bdiv(\yhat,\ybar) 
  - \frac{\delta}{4} \bigl\|\nabla f^*(\yopt)-\nabla f^*(\yhat)\bigr\|^2\\
&& +\, \yopt^T A\xhat - \yhat^T A\xopt + \ytld^T A\xopt 
  - \ytld^T A\xhat - \yopt^T A\xtld + \yhat^T A \xtld .
\end{eqnarray*}
We can simplify the inner product terms as
\[
\yopt^T A\xhat - \yhat^T A\xopt + \ytld^T A\xopt 
- \ytld^T A\xhat - \yopt^T A\xtld + \yhat^T A \xtld
= (\yhat-\ytld)^TA(\xhat-\xopt) - (\yhat-\yopt)^T A(\xhat-\xtld).
\]
Rearranging terms on the two sides of the inequality, we have
\begin{eqnarray*}
\frac{1}{2\tau}\|\xopt-\xbar\|^2 + \frac{1}{\sigma}\Bdiv(\yopt,\ybar)
&\geq& \Lagr(\xhat,\yopt) - \Lagr(\xopt,\yhat) \\
&& +\, \frac{1}{2}\Bigl(\frac{1}{\tau}+\lambda\Bigr) \|\xopt-\xhat\|^2
+ \Bigl(\frac{1}{\sigma}+\frac{\nu}{2}\Bigr) \Bdiv(\yopt,\yhat) \\
&& +\,\frac{1}{2\tau}\|\xhat-\xbar\|^2 + \frac{1}{\sigma}\Bdiv(\yhat,\ybar) 
  + \frac{\delta}{4} \bigl\|\nabla f^*(\yopt)-\nabla f^*(\yhat)\bigr\|^2\\
&& +\, (\yhat-\yopt)^T A(\xhat-\xtld) - (\yhat-\ytld)^TA(\xhat-\xopt).
\end{eqnarray*}
Applying the substitutions in~\eqref{eqn:y-first-subs} yields
\begin{eqnarray}
\frac{1}{2\tau}\|\xopt-\xt\|^2 + \frac{1}{\sigma}\Bdiv(\yopt,\yt)
&\geq& \Lagr(\xtp,\yopt) - \Lagr(\xopt,\ytp) \nonumber \\
&& +\, \frac{1}{2}\Bigl(\frac{1}{\tau}+\lambda\Bigr) \|\xopt-\xtp\|^2
+ \Bigl(\frac{1}{\sigma}+\frac{\nu}{2}\Bigr) \Bdiv(\yopt,\ytp) \nonumber \\
&& +\,\frac{1}{2\tau}\|\xtp-\xt\|^2 + \frac{1}{\sigma}\Bdiv(\ytp,\yt) 
  + \frac{\delta}{4} \bigl\|\nabla f^*(\yopt)-\nabla f^*(\ytp)\bigr\|^2
  \nonumber \\
&& +\, (\ytp-\yopt)^T A\bigl(\xtp-(\xt+\theta(\xt-\xtm)\bigr).
\label{eqn:PD-bound-y-temp}
\end{eqnarray}
We can rearrange the inner product term in~\eqref{eqn:PD-bound-y-temp} as
\begin{eqnarray*}
&& (\ytp-\yopt)^T A\bigl(\xtp-(\xt+\theta(\xt-\xtm)\bigr) \\
&=& (\ytp-\yopt)^T A(\xtp-\xt) - \theta(\yt-\yopt)^T A(\xt-\xtm)
 - \theta(\ytp-\yt)^T A(\xt-\xtm) .
\end{eqnarray*}
Using the optimality conditions in~\eqref{eqn:saddle-y-optimality}
and~\eqref{eqn:PD-y-optimality}, we can also bound 
$\|\nabla f^*(\yopt)-\nabla f^*(\ytp)\|^2$:
\begin{eqnarray}
&& \bigl\|\nabla f^*(\yopt)-\nabla f^*(\ytp) \bigr\|^2 \nonumber\\
&=& \Bigl\| A\xopt - A\bigl(\xt+\theta(\xt-\xtm)\bigr) + \frac{1}{\sigma}
    \bigl(\nabla h(\ytp) - \nabla h(\yt)\bigr) \Bigr\|^2 \nonumber\\ 
&\geq& \Bigl(1-\frac{1}{\alpha}\Bigr) \bigl\|A(\xopt-\xt)\bigr\|^2
 - (\alpha-1) \Bigl\| \theta A (\xt-\xtm)\bigr) - \frac{1}{\sigma}
    \bigl(\nabla h(\ytp) - \nabla h(\yt)\bigr) \Bigr\|^2, \nonumber
\label{eqn:conj-grad-bound}
\end{eqnarray}
where $\alpha>1$. 
With the definition $\mu=\sqrt{\lambdamin(A^T A)}$, we also have
$\|A(\xopt-\xt)\|^2\geq\mu^2\|\xopt-\xt\|^2$.
Combining them with the inequality~\eqref{eqn:PD-bound-y-temp} leads to
\begin{eqnarray}
&& \frac{1}{2\tau}\|\xopt-\xt\|^2 + \frac{1}{\sigma}\Bdiv(\yopt,\yt)
+ \theta(\yt-\yopt)^T A(\xt-\xtm) \nonumber \\
&\geq& \Lagr(\xtp,\yopt) - \Lagr(\xopt,\ytp) \nonumber \\
&& +\, \frac{1}{2}\Bigl(\frac{1}{\tau}+\lambda\Bigr) \|\xopt-\xtp\|^2
+ \Bigl(\frac{1}{\sigma}+\frac{\nu}{2}\Bigr) \Bdiv(\yopt,\ytp) 
  + (\ytp-\yopt)^T A(\xtp-\xt) \nonumber \\
&& +\,\frac{1}{2\tau}\|\xtp-\xt\|^2 + \frac{1}{\sigma}\Bdiv(\ytp,\yt) 
 - \theta(\ytp-\yt)^T A(\xt-\xtm) \nonumber \\
&& + \Bigl(1-\frac{1}{\alpha}\Bigr)\frac{\delta\mu^2}{4}\|\xopt-\xt\|^2
 - (\alpha-1)\frac{\delta}{4}\Bigl\| \theta A(\xt-\xtm)\bigr)-\frac{1}{\sigma}
    \bigl(\nabla h(\ytp) - \nabla h(\yt)\bigr) \Bigr\|^2 .
\label{eqn:PD-bound-y-first}
\end{eqnarray}

\section{Proof of Theorem \ref{thm:BPD-convergence}}
\label{sec:analysis-Euclidean}

Let the kernel function be $h(y)=(1/2)\|y\|^2$. 
In this case, we have $\Bdiv(y',y)=(1/2)\|y'-y\|^2$ and $\nabla h(y)=y$. 
Moreover, $\gamma'=\delta'=1$ and $\nu=\gamma$.
Therefore, the inequality~\eqref{eqn:PD-bound-y-first} becomes
\begin{eqnarray}
&& \frac{1}{2}\biggl(\frac{1}{\tau}-\Bigl(1-\frac{1}{\alpha}\Bigr)
\frac{\delta\mu^2}{2}\biggr) \|\xopt-\xt\|^2 + \frac{1}{2\sigma}\|\yopt-\yt\|^2
  + \theta(\yt-\yopt)^T A(\xt-\xtm) \nonumber\\
&\geq& \Lagr(\xtp,\yopt) - \Lagr(\xopt,\ytp) \nonumber \\
&& +\, \frac{1}{2}\Bigl(\frac{1}{\tau}+\lambda\Bigr) \|\xopt-\xtp\|^2
+ \frac{1}{2}\Bigl(\frac{1}{\sigma}+\frac{\gamma}{2}\Bigr) \|\yopt-\ytp\|^2 
  + (\ytp-\yopt)^T A(\xtp-\xt) \nonumber \\
&& +\,\frac{1}{2\tau}\|\xtp-\xt\|^2 
 + \underline{\frac{1}{2\sigma}\|\ytp-\yt\|^2-\theta(\ytp-\yt)^T A(\xt-\xtm)}
 \nonumber \\
&& -\,(\alpha-1)\frac{\delta}{4}\Bigl\|\theta A(\xt-\xtm)\bigr)
   -\frac{1}{\sigma} (\ytp - \yt) \Bigr\|^2 .
\label{eqn:PD-bound-y-Euclid}
\end{eqnarray}
Next we derive another form of the underlined items above:
\begin{eqnarray*}
&& \frac{1}{2\sigma}\|\ytp-\yt\|^2 - \theta (\ytp-\yt)^T A (\xt-\xtm) \\
&=& \frac{\sigma}{2}\left( \frac{1}{\sigma^2}\|\ytp-\yt\|^2
  - \frac{\theta}{\sigma}(\ytp-\yt)^T A (\xt-\xtm) \right) \\
&=& \frac{\sigma}{2}\left( 
  \Bigl\|\theta A(\xt-\xtm)-\frac{1}{\sigma}(\ytp-\yt)\Bigr\|^2 
  -\theta^2 \|A(\xt-\xtm)\|^2 \right) \\
&\geq& \frac{\sigma}{2} 
  \Bigl\|\theta A(\xt-\xtm)-\frac{1}{\sigma}(\ytp-\yt)\Bigr\|^2 
  -\frac{\sigma\theta^2 L^2}{2} \|\xt-\xtm\|^2 ,
\end{eqnarray*}
where in the last inequality we used $\|A\|\leq L$ and hence
$\|A(\xt-\xtm)\|^2 \leq L^2\|\xt-\xtm\|^2$.
Combining with inequality~\eqref{eqn:PD-bound-y-Euclid}, we have
\begin{eqnarray}
&& \frac{1}{2}\left(\frac{1}{\tau}-\Bigl(1-\frac{1}{\alpha}\Bigr)
  \frac{\delta\mu^2}{2}\right)\|\xt-\xopt\|^2 + \frac{1}{2\sigma}
  \|\yt-\yopt\|^2 + \theta (\yt-\yopt)^T A(\xt-\xtm)  
  +\frac{\sigma\theta^2 L^2}{2} \|\xt-\xtm\|^2  \nonumber \\
&\geq& \Lagr(\xtp,\yopt) - \Lagr(\xopt,\ytp) \nonumber \\
&&+\,\frac{1}{2}\left(\frac{1}{\tau}+\lambda\right) \|\xtp-\xopt\|^2 
  + \frac{1}{2}\left(\frac{1}{\sigma}+\frac{\gamma}{2}\right)\|\ytp-\yopt\|^2
  + (\ytp-\yopt)^T A(\xtp-\xt) + \frac{1}{2\tau}\|\xtp-\xt\|^2  \nonumber \\
&& + \left(\frac{\sigma}{2}-(\alpha-1) \frac{\delta}{4}\right) 
  \Bigl\|\theta A(\xt-\xtm)-\frac{1}{\sigma}(\ytp-\yt)\Bigr\|^2 .
  \label{eqn:PD-complete-squares}
\end{eqnarray}
We can remove the last term in the above inequality as long as
its coefficient is nonnegative, i.e.,  
\[
  \frac{\sigma}{2}-(\alpha-1) \frac{\delta}{4} \geq 0.
\]
In order to maximize $1-1/\alpha$, we take the equality and solve for 
the largest value of~$\alpha$ allowed, which results in
\[
  \alpha = 1 + \frac{2\sigma}{\delta} , \qquad
  1-\frac{1}{\alpha} 
  = \frac{2\sigma}{2\sigma+\delta}.
\]
Applying these values in~\eqref{eqn:PD-complete-squares} gives
\begin{eqnarray}
&& \frac{1}{2}\left(\frac{1}{\tau}-\frac{\sigma\delta\mu^2}{2\sigma+\delta}
  \right) \|\xt-\xopt\|^2 + \frac{1}{2\sigma}
  \|\yt-\yopt\|^2 + \theta (\yt-\yopt)^T A(\xt-\xtm)  
  +\frac{\sigma\theta^2 L^2}{2} \|\xt-\xtm\|^2  \nonumber \\
&\geq& \Lagr(\xtp,\yopt) - \Lagr(\xopt,\ytp) \nonumber \\
&& + \frac{1}{2}\left(\frac{1}{\tau}+\lambda\right) \|\xtp-\xopt\|^2 
  + \frac{1}{2}\left(\frac{1}{\sigma}+\frac{\gamma}{2}\right)\|\ytp-\yopt\|^2
  + (\ytp-\yopt)^T A(\xtp-\xt) + \frac{1}{2\tau}\|\xtp-\xt\|^2. \nonumber \\
  \label{eqn:PD-xy-compare}
\end{eqnarray}

We use $\Dtp$ to denote the last row in~\eqref{eqn:PD-xy-compare}.
Equivalently, we define
\begin{eqnarray*}
\Dt 
&=& \frac{1}{2}\Bigl(\frac{1}{\tau}+\lambda\Bigr) \|\xopt-\xt\|^2
  + \frac{1}{2}\Bigl(\frac{1}{\sigma}+\frac{\gamma}{2}\Bigr) \|\yopt-\yt\|^2
  + (\yt-\yopt)^T A(\xt-\xtm) + \frac{1}{2\tau}\|\xt-\xtm\|^2 \\
&=& \frac{1}{2}\Bigl(\frac{1}{\tau}+\lambda\Bigr) \|\xopt-\xt\|^2
  + \frac{\gamma}{4}\|\yopt-\yt\|^2 
  + \frac{1}{2}\left[\begin{array}{c} \xt-\xtm\\ \yopt-\yt \end{array}\right]^T
  \left[\begin{array}{cc} \frac{1}{\tau}I & -A^T\\ -A & \frac{1}{\sigma}
  \end{array} \right]
  \left[\begin{array}{c} \xt-\xtm \\ \yopt-\yt \end{array}\right].
\end{eqnarray*}
The quadratic form in the last term is nonnegative if the matrix
\[
M = \left[\begin{array}{cc} \frac{1}{\tau}I & -A^T\\ -A & \frac{1}{\sigma}
  \end{array} \right]
\]
is positive semidefinite, for which a sufficient condition is
$\tau\sigma\leq 1/L^2$.
Under this condition,
\begin{equation}\label{eqn:Delta-nonneg-Euclid}
\Dt ~\geq~ \frac{1}{2}\Bigl(\frac{1}{\tau}+\lambda\Bigr) \|\xopt-\xt\|^2
+ \frac{\gamma}{4} \|\yopt-\yt\|^2 ~\geq~ 0.
\end{equation}
If we can to choose $\tau$ and $\sigma$ so that
\begin{equation}\label{eqn:tau-sigma-cond-Euclid}
\frac{1}{\tau}-\frac{\sigma \delta\mu^2}{2\sigma+\delta}
\leq \theta\Bigl(\frac{1}{\tau}+\lambda\Bigr), \qquad
\frac{1}{\sigma} \leq \theta \Bigl(\frac{1}{\sigma}+\frac{\gamma}{2}\Bigr),
\qquad
\frac{\sigma\theta^2 L^2}{2} \leq \theta\frac{1}{2\tau} ,
\end{equation}
then, according to~\eqref{eqn:PD-xy-compare}, we have
\[
  \Dtp + \Lagr(\xtp,\yopt) - \Lagr(\xopt,\ytp) ~\leq~ \theta \Dt .
\]
Because $\Dt\geq0$ and $\Lagr(\xt,\yopt) - \Lagr(\xopt,\yt) \geq 0$ 
for any $t\geq 0$, we have
\[
  \Dtp ~\leq~ \theta \Dt,
\]
which implies
\[
  \Dt ~\leq~ \theta^t \Dini
\]
and
\[
  \Lagr(\xt,\yopt) - \Lagr(\xopt,\yt) ~\leq~ \theta^t \Dini .
\]
Let $\theta_x$ and $\theta_y$ be two contraction factors determined by
the first two inequalities in~\eqref{eqn:tau-sigma-cond-Euclid}, i.e.,
\begin{eqnarray*}
    \theta_x 
    &=& \biggl(\frac{1}{\tau}-\frac{\sigma \delta\mu^2}{2\sigma+\delta}\biggr)
    \bigg/ \Bigl(\frac{1}{\tau}+\lambda\Bigr)
    ~=~ \biggl(1-\frac{\tau\sigma\delta\mu^2}{2\sigma+\delta}\biggr)\frac{1}{1+\tau\lambda}, \\ 
    \theta_y
    &=&\frac{1}{\sigma} \bigg/\Bigl(\frac{1}{\sigma}+\frac{\gamma}{2}\Bigr)
    ~=~ \frac{1}{1+\sigma\gamma/2}.
\end{eqnarray*}
Then we can let $\theta=\max\{\theta_x,\theta_y\}$. 
We note that any $\theta<1$ would satisfy the last condition 
in~\eqref{eqn:tau-sigma-cond-Euclid} provided that
\[
  \tau\sigma ~=~ \frac{1}{L^2} ,
\]
which also makes the matrix~$M$ positive semidefinite and thus
ensures the inequality~\eqref{eqn:Delta-nonneg-Euclid}. 

Among all possible pairs $\tau,\sigma$ that satisfy $\tau\sigma=1/L^2$,
we choose 
\begin{equation}\label{eqn:tau-sigma-Euclid}
  \tau = \frac{1}{L}\sqrt{\frac{\gamma}{\lambda+\delta\mu^2}}, \qquad
  \sigma = \frac{1}{L}\sqrt{\frac{\lambda+\delta\mu^2}{\gamma}},
\end{equation}
which give the desired results of Theorem~\ref{thm:BPD-convergence}.

\section{Proof of Theorem~\ref{thm:DF-BPD-convergence}}
\label{sec:proof_dfbpd}

If we choose $h=f^*$, then 
\begin{itemize}
  \item $h$ is $\gamma$-strongly convex and $1/\delta$-smooth, 
    i.e., $\gamma'=\gamma$ and $\delta'=\delta$;
\item $f^*$ is $1$-strongly convex with respect to~$h$, i.e., $\nu=1$.
\end{itemize}
For convenience, we repeat inequality~\eqref{eqn:PD-bound-y-first} here:
\begin{eqnarray}
&& \frac{1}{2\tau}\|\xopt-\xt\|^2 + \frac{1}{\sigma}\Bdiv(\yopt,\yt)
+ \theta(\yt-\yopt)^T A(\xt-\xtm) \nonumber \\
&\geq& \Lagr(\xtp,\yopt) - \Lagr(\xopt,\ytp) \nonumber \\
&& +\, \frac{1}{2}\Bigl(\frac{1}{\tau}+\lambda\Bigr) \|\xopt-\xtp\|^2
+ \Bigl(\frac{1}{\sigma}+\frac{\nu}{2}\Bigr) \Bdiv(\yopt,\ytp) 
  + (\ytp-\yopt)^T A(\xtp-\xt) \nonumber \\
&& +\,\frac{1}{2\tau}\|\xtp-\xt\|^2 + \frac{1}{\sigma}\Bdiv(\ytp,\yt) 
 - \theta(\ytp-\yt)^T A(\xt-\xtm) \nonumber \\
&& + \Bigl(1-\frac{1}{\alpha}\Bigr)\frac{\delta\mu^2}{4}\|\xopt-\xt\|^2
 - (\alpha-1)\frac{\delta}{4}\Bigl\| \theta A(\xt-\xtm)\bigr)-\frac{1}{\sigma}
    \bigl(\nabla h(\ytp) - \nabla h(\yt)\bigr) \Bigr\|^2 .
\label{eqn:PD-bound-y-repeat}
\end{eqnarray}
We first bound the Bregman divergence $\Bdiv(\ytp,\yt)$ 
using the assumption that the kernel $h$ is $\gamma$-strongly convex and 
$1/\delta$-smooth. Using similar arguments as in the proof of
Lemma~\ref{lem:sc-min-3point}, we have for any $\rho\in[0,1]$, 
\begin{eqnarray}
D(\ytp,\yt) 
&=& h(\ytp)-h(\yt)-\langle \nabla h(\yt), \ytp-\yt \rangle \nonumber \\
&\geq& (1-\rho) \frac{\gamma}{2} \|\ytp-\yt\|^2 
 + \rho\frac{\delta}{2} \bigl\|\nabla h(\ytp) - \nabla h(\yt) \bigr\|^2.
\label{eqn:Bdiv-ytp-yp}
\end{eqnarray}
For any $\beta>0$, we can lower bound the inner product term
\[
 -\theta(\ytp-\yt)^T A(\xt-\xtm) \geq - \frac{\beta}{2}\|\ytp-\yt\|^2 
 - \frac{\theta^2 L^2}{2\beta}\|\xt-\xtm\|^2.
\]
In addition, we have
\[
\Bigl\| \theta A(\xt-\xtm)\bigr)-\frac{1}{\sigma}
    \bigl(\nabla h(\ytp) - \nabla h(\yt)\bigr) \Bigr\|^2 
~\leq~ 2\theta^2 L^2 \|\xt-\xtm\|^2
+ \frac{2}{\sigma^2}\bigl\|\nabla h(\ytp) - \nabla h(\yt) \bigr\|^2.
\]
Combining these bounds with~\eqref{eqn:PD-bound-y-repeat} 
and~\eqref{eqn:Bdiv-ytp-yp} with $\rho=1/2$, we arrive at
\begin{eqnarray}
&& \frac{1}{2}\biggl(\frac{1}{\tau}-\Bigl(1-\frac{1}{\alpha}\Bigr)
\frac{\delta\mu^2}{2}\biggr) \|\xopt-\xt\|^2 + \frac{1}{\sigma}\Bdiv(\yopt,\yt)
  + \theta(\yt-\yopt)^T A(\xt-\xtm) \nonumber\\
&& +\Bigl(\frac{\theta^2 L^2}{2\beta} 
+ (\alpha-1)\frac{\delta\theta^2 L^2}{2}\Bigr) \|\xt-\xtm\|^2 \nonumber \\
&\geq& \Lagr(\xtp,\yopt) - \Lagr(\xopt,\ytp) \nonumber \\
&& +\, \frac{1}{2}\Bigl(\frac{1}{\tau}+\lambda\Bigr) \|\xopt-\xtp\|^2
+ \Bigl(\frac{1}{\sigma}+\frac{1}{2}\Bigr) \Bdiv(\yopt,\ytp) 
  + (\ytp-\yopt)^T A(\xtp-\xt) \nonumber \\
&& + \Bigl(\frac{\gamma}{4\sigma}-\frac{\beta}{2}\Bigr)\|\ytp-\yt\|^2 
 + \Bigl(\frac{\delta}{4\sigma} - \frac{(\alpha-1)\delta}{2\sigma^2}
  \Bigr) \bigl\|\nabla h(\ytp)-\nabla h(\yt)\bigr\|^2 \nonumber\\
&& +\,\frac{1}{2\tau}\|\xtp-\xt\|^2 .
\label{eqn:PD-bound-y-full}
\end{eqnarray}

We choose $\alpha$ and $\beta$ in~\eqref{eqn:PD-bound-y-full} to zero out 
the coefficients of $\|\ytp-\yt\|^2$ and $\|\nabla h(\ytp)-\nabla h(\yt)\|^2$:
\[
  \alpha = 1 + \frac{\sigma}{2}, 
  \qquad
  \beta = \frac{\gamma}{2\sigma}.
\]
Then the inequality~\eqref{eqn:PD-bound-y-full} becomes
\begin{eqnarray}
&& \frac{1}{2}\biggl(\frac{1}{\tau}-\frac{\sigma\delta\mu^2}{4+2\sigma}
  \biggr) \|\xopt-\xt\|^2 + \frac{1}{\sigma}\Bdiv(\yopt,\yt)
  + \theta(\yt-\yopt)^T A(\xt-\xtm) \nonumber\\
&& +\biggl(\frac{\sigma\theta^2 L^2}{\gamma} 
+ \frac{\delta\sigma\theta^2 L^2}{4} \biggr)\|\xt-\xtm\|^2 \nonumber \\
&\geq& \Lagr(\xtp,\yopt) - \Lagr(\xopt,\ytp) \nonumber \\
&& +\, \frac{1}{2}\Bigl(\frac{1}{\tau}+\lambda\Bigr) \|\xopt-\xtp\|^2
+ \biggl(\frac{1}{\sigma}+\frac{1}{2}\biggr) \Bdiv(\yopt,\ytp) 
  + (\ytp-\yopt)^T A(\xtp-\xt) \nonumber \\
&& +\,\frac{1}{2\tau}\|\xtp-\xt\|^2 . \nonumber
\end{eqnarray}
The coefficient of $\|\xt-\xtm\|^2$ can be bounded as
\[
  \frac{\sigma\theta^2 L^2}{\gamma} + \frac{\delta\sigma\theta^2 L^2}{4}
=\Bigl(\frac{1}{\gamma}+\frac{\delta}{4}\Bigr)\sigma\theta^2 L^2
= \frac{4+\gamma\delta}{4\gamma}\sigma\theta^2 L^2 
< \frac{2 \sigma\theta^2 L^2}{\gamma},
\]
where in the inequality we used $\gamma\delta\leq 1$.
Therefore we have
\begin{eqnarray}
&& \frac{1}{2}\biggl(\frac{1}{\tau}-\frac{\sigma\delta\mu^2}{4+2\sigma}
  \biggr) \|\xopt-\xt\|^2 + \frac{1}{\sigma}\Bdiv(\yopt,\yt)
  + \theta(\yt-\yopt)^T A(\xt-\xtm) 
  + \frac{2\sigma\theta^2 L^2}{\gamma} \|\xt-\xtm\|^2 \nonumber \\
&\geq& \Lagr(\xtp,\yopt) - \Lagr(\xopt,\ytp) \nonumber \\
&& +\, \frac{1}{2}\Bigl(\frac{1}{\tau}+\lambda\Bigr) \|\xopt-\xtp\|^2
+ \biggl(\frac{1}{\sigma}+\frac{1}{2}\biggr) \Bdiv(\yopt,\ytp) 
  + (\ytp-\yopt)^T A(\xtp-\xt) 
  + \frac{1}{2\tau}\|\xtp-\xt\|^2.  \nonumber
\end{eqnarray}
We use $\Dtp$ to denote the last row of the above inequality. 
Equivalently, we define
\[
\Dt = 
\frac{1}{2}\Bigl(\frac{1}{\tau}+\lambda\Bigr) \|\xopt-\xt\|^2
+ \biggl(\frac{1}{\sigma}+\frac{1}{2}\biggr) \Bdiv(\yopt,\yt) 
+ (\yt-\yopt)^T A(\xt-\xtm) + \frac{1}{2\tau}\|\xt-\xtm\|^2  .
\]
Since $h$ is $\gamma$-strongly convex, we have 
$\Bdiv(\yopt,\yt)\geq\frac{\gamma}{2}\|\yopt-\yt\|^2$, and thus
\begin{eqnarray*}
\Dt 
&\geq& \frac{1}{2}\Bigl(\frac{1}{\tau}+\lambda\Bigr) \|\xopt-\xt\|^2
  + \frac{1}{2}\Bdiv(\yopt,\yt) + \frac{\gamma}{2\sigma}\|\yt-\yopt\|^2
  + (\yt-\yopt)^T A(\xt-\xtm) + \frac{1}{2\tau}\|\xt-\xtm\|^2 \\ 
&=& \frac{1}{2}\Bigl(\frac{1}{\tau}+\lambda\Bigr) \|\xopt-\xt\|^2
  + \frac{1}{2}\Bdiv(\yopt,\yt) 
  + \frac{1}{2}\left[\begin{array}{c} \xt-\xtm\\ \yopt-\yt \end{array}\right]^T
  \left[\begin{array}{cc} \frac{1}{\tau}I & -A^T\\ -A & \frac{\gamma}{\sigma}
  \end{array} \right]
  \left[\begin{array}{c} \xt-\xtm \\ \yopt-\yt \end{array}\right].
\end{eqnarray*}
The quadratic form in the last term is nonnegative if
$\tau\sigma\leq\gamma/L^2$.
Under this condition,
\begin{equation}\label{eqn:Delta-nonneg}
\Dt ~\geq~ \frac{1}{2}\Bigl(\frac{1}{\tau}+\lambda\Bigr) \|\xopt-\xt\|^2
+ \frac{1}{2} \Bdiv(\yopt,\yt) ~\geq~ 0.
\end{equation}
If we can to choose $\tau$ and $\sigma$ so that
\begin{equation}\label{eqn:tau-sigma-cond-nu1}
\frac{1}{\tau}-\frac{\sigma \delta\mu^2}{4+2\sigma}
\leq \theta\Bigl(\frac{1}{\tau}+\lambda\Bigr), \qquad
\frac{1}{\sigma} \leq \theta \Bigl(\frac{1}{\sigma}+\frac{1}{2}\Bigr),
\qquad
\frac{2\sigma\theta^2 L^2}{\gamma} \leq \theta\frac{1}{2\tau} ,
\end{equation}
then we have
\[
  \Dtp + \Lagr(\xtp,\yopt) - \Lagr(\xopt,\ytp) ~\leq~ \theta \Dt .
\]
Because $\Dt\geq 0$ and $\Lagr(\xt,\yopt) - \Lagr(\xopt,\yt) \geq 0$ 
for any $t\geq 0$, we have
\[
  \Dtp ~\leq~ \theta \Dt,
\]
which implies
\[
  \Dt ~\leq~ \theta^t \Dini
\]
and
\[
  \Lagr(\xt,\yopt) - \Lagr(\xopt,\yt) ~\leq~ \theta^t \Dini .
\]

To satisfy the last condition in~\eqref{eqn:tau-sigma-cond-nu1} and also ensure
the inequality~\eqref{eqn:Delta-nonneg}, it suffices to have
\[
  \tau\sigma ~\leq~ \frac{\gamma}{4 L^2} .
\]
We choose 
\[
  \tau = \frac{1}{2L}\sqrt{\frac{\gamma}{\lambda+\delta\mu^2}}, \qquad
  \sigma = \frac{1}{2L}\sqrt{\gamma(\lambda+\delta\mu^2)}.
\]
With the above choice and assuming $\gamma(\lambda+\delta\mu^2)\ll L^2$, 
we have
\[
  \theta_y = \frac{\frac{1}{\sigma}}{\frac{1}{\sigma}+\frac{1}{2}}
  =\frac{1}{1+\sigma/2}
  =\frac{1}{1+\sqrt{\gamma(\lambda+\delta\mu^2)}/(4L)}
  \approx 1 - \frac{\sqrt{\gamma(\lambda+\delta\mu^2)}}{4L}.
\]
For the contraction factor over the primal variables, we have
\[
  \theta_x = \frac{\frac{1}{\tau}-\frac{\sigma \delta\mu^2}{4+2\sigma}}{
             \frac{1}{\tau}+\lambda}
  = \frac{1-\frac{\tau\sigma \delta\mu^2}{4+2\sigma}}{1+\tau\lambda}
  = \frac{1-\frac{\gamma\delta\mu^2}{4(4+2\sigma)L^2}}{1+\tau\lambda}
  \approx 1-\frac{\gamma\delta\mu^2}{16 L^2}
  - \frac{\lambda}{2L}\sqrt{\frac{\gamma}{\lambda+\delta\mu^2}}.
\]
This finishes the proof of Theorem~\ref{thm:DF-BPD-convergence}.

\section{Proof of Theorem \ref{thm:SPDC-convergence}}
\label{sec:proof_spdc}

\newcommand{\EE}{\mathbb{E}} 
We consider the SPDC algorithm in the Euclidean case with $h(x)=(1/2)\|x\|^2$.
The corresponding batch case analysis is given in 
Section~\ref{sec:analysis-Euclidean}. 
For each $i = 1,\ldots,n$, let $\yitld$ be
\[
    \yitld = \arg\min_{y} \left\{ \phi_i^*(y) + \frac{(y - \yit)^2}{2 \sigma} - y \langle a_i, \xtldt \rangle \right\}.
\]
Based on the first-order optimality condition, we have
\[
\langle a_i, \xtldt \rangle - \frac{(\yitld - \yit)}{\sigma} \in \phi_i^{*'}(\yitld).
\]
Also, since $\yiopt$ minimizes $\phi_i^*(y) - y \langle a_i, x^* \rangle$, 
we have
\[
\langle a_i, x^* \rangle \in \phi_i^{*'}(\yiopt).
\]
By Lemma \ref{lem:sc-min-3point} with $\rho = 1/2$, we have
\begin{align*}
- \yiopt \langle a_i, \xtldt \rangle + \phi_i^*(\yiopt) + 
\frac{(\yit - \yiopt)^2}{2\sigma } \geq& \left( \frac{1}{\sigma} + \frac{\gamma}{2} \right) \frac{(\yitld - \yiopt)^2 }{2} + \phi^*_i(\yitld) - \yitld \langle a_i, \xtldt \rangle \\
&+\frac{(\yitld - \yit)^2}{2 \sigma} + \frac{\delta}{4} ( \phi_i^{*'}(\yitld) - \phi_i^{*'}(\yiopt) )^2,
\end{align*}
and re-arranging terms, we get
\begin{align}
\frac{(\yit - \yiopt)^2}{2\sigma }  \geq& \left( \frac{1}{\sigma} + \frac{\gamma}{2} \right) \frac{(\yitld - \yiopt)^2 }{2} + \frac{(\yitld - \yit)^2}{2 \sigma}  - (\yitld - \yiopt) \langle a_i, \xtldt \rangle +  ( \phi_i^*(\yitld) -  \phi_i^*(\yiopt) ) \nonumber \\
&+ \frac{\delta}{4} ( \phi_i^{*'}(\yitld) - \phi_i^{*'}(\yiopt) )^2.
\label{eqn:y-main-ineq}
\end{align}
Notice that
\begin{align*}
\EE [ \yitp ] =& \frac{1}{n} \cdot \yitld + \frac{n-1}{n} \cdot \yit, \\
\EE [ (\yitp - \yiopt)^2 ] =& \frac{ (\yitld - \yiopt)^2 }{n} + \frac{ (n-1)(\yit - \yiopt)^2 }{n}, \\
\EE [ (\yitp - \yit)^2 ] =& \frac{( \yitld - \yit)^2}{n} , \\
\EE [ \phi_i^*(\yitp) ] =& \frac{1}{n}\cdot  \phi_i^*(\yitld) + \frac{n-1}{n} \cdot \phi_i^*(\yit).
\end{align*}
Plug the above relations into \eqref{eqn:y-main-ineq} and divide both sides by $n$, we have
\begin{align*}
\left(  \frac{1}{2 \sigma} + \frac{(n-1)\gamma}{4n} \right) (\yit - \yiopt)^2 \geq& 
\left(  \frac{1}{2 \sigma} + \frac{\gamma}{4} \right) \EE [ (\yitp - \yiopt)^2 ] + \frac{1}{2 \sigma} \EE [ (\yitp - \yit)^2 ] \\
&- \left(  \EE [ ( \yitp - \yit) ] + \frac{1}{n} (\yit - \yiopt)  \right) \langle a_i, \xtldt \rangle \\
&+ \EE [ \phi_i^*(\yitp) ] - \phi_i^*(\yit) + \frac{1}{n} ( \phi_i^*(\yit) - \phi_i^*(\yiopt) ) \\
&+ \frac{\delta}{4n} \left( \langle a_i, \xtldt - \xopt \rangle - \frac{(\yitld - \yit)}{\sigma}  \right)^2,
\end{align*}
and summing over $i=1,\ldots,n$, we get
\begin{align*}
\left(  \frac{1}{2 \sigma} + \frac{(n-1)\gamma}{4n} \right) \| \yt - \yopt \|^2 \geq& 
\left(  \frac{1}{2 \sigma} + \frac{\gamma}{4}  \right) \EE [ \| \ytp - \yopt \|^2 ]
+ \frac{\EE [ \| \ytp - \yt \|^2 ]}{2 \sigma}  \\
&+ \phi_k^*(\yktp) - \phi_k^*(\ykt) + \frac{1}{n} \sum_{i=1}^n ( \phi_i^*(\yit) - \phi_i^*(\yiopt) )\\
&- \left\langle  n(\utp - \ut) + ( \ut - u^*), \xtldt \right\rangle \\
&+ \frac{\delta}{4n} \left\| A (x^* - \xtldt ) + \frac{ ( \ytld - \yt ) } {\sigma} \right\|^2,
\end{align*}
where
\[
\ut = \frac{1}{n} \sum_{i=1}^n \yit a_i, \quad \utp = \frac{1}{n} \sum_{i=1}^n \yitp a_i, \quad {\rm and } \quad u^* = \frac{1}{n} \sum_{i=1}^n \yiopt a_i.
\]

On the other hand, since $\xtp$ minimizes the $\frac{1}{\tau} + \lambda$-strongly convex objective
\[
g(x) + \left \langle \ut + n(\utp - \ut), x \right \rangle + \frac{\| x - \xt \|^2}{2 \tau},
\]
we can apply Lemma \ref{lem:sc-min-3point} with $\rho = 0$ to obtain
\begin{align*}
&g(\xopt) + \langle \ut + n( \utp - \ut), \xopt \rangle + \frac{ \| \xt - \xopt \|^2 }{2 \tau} \\
&\geq g(\xtp) + \langle \ut + n( \utp - \ut), \xtp \rangle + \frac{  \| \xtp - \xt \|^2  }{2 \tau} + \left( \frac{1}{2\tau} + \frac{\lambda}{2} \right) \| \xtp - \xopt \|^2 ,
\end{align*}
and re-arranging terms we get
\begin{align*}
\frac{ \| \xt - \xopt \|^2 }{2 \tau} \geq& \left( \frac{1}{2\tau} + \frac{\lambda}{2} \right) \EE [ \| \xtp - \xopt \|^2 ] + \frac{ \EE [ \| \xtp - \xt \|^2 ] }{2 \tau} 
+ \EE [ g(\xtp) - g(\xopt) ] \\
&+ \EE [ \langle \ut + n( \utp - \ut), \xtp - \xopt \rangle ] .
\end{align*}
Also notice that
\begin{align*}
&\Lagr(\xtp,\yopt) - \Lagr(\xopt,\yopt) + n( \Lagr(\xopt,\yopt) - \Lagr(\xopt,\ytp) ) - (n-1) (\Lagr(\xopt,\yopt) - \Lagr(\xopt,\yt) ) \\
&= \frac{1}{n} \sum_{i=1}^n ( \phi_i^*(\yit) - \phi_i^*(\yopt) )
+ (\phi_k^*(\yktp) - \phi_k^*(\ykt) ) + g(\xtp) - g(\xopt) \\
&+ \langle \uopt, \xtp \rangle - \langle \ut, \xopt \rangle + n \langle \ut - \utp, \xopt \rangle.
\end{align*}
Combining everything together, we have
\begin{align*}
&\frac{ \| \xt - \xopt \|^2 }{2 \tau} + \left(  \frac{1}{2 \sigma} + \frac{(n-1)\gamma}{4n} \right) \| \yt - \yopt \|^2  + (n-1) (\Lagr(\xopt,\yopt) - \Lagr(\xopt,\yt) ) \\
&\geq \left( \frac{1}{2\tau} + \frac{\lambda}{2} \right) \EE [ \| \xtp - \xopt \|^2 ]
+ \left(  \frac{1}{2 \sigma} + \frac{\gamma}{4}  \right) \EE [ \| \ytp - \yopt \|^2 ] + 
\frac{ \EE [ \| \xtp - \xt \|^2 ] }{2 \tau}  + \frac{\EE [ \| \ytp - \yt \|^2 ]}{2 \sigma}\\
&+ \EE [ \Lagr(\xtp,\yopt) - \Lagr(\xopt,\yopt) + n( \Lagr(\xopt,\yopt) - \Lagr(\xopt,\ytp) ) ]   \\
&+ \EE [ \langle  \ut - u^* + n( \utp - \ut), \xtp - \bar x^t \rangle] + \frac{\delta}{4n} \left\| A (x^* - \xtldt ) + \frac{ ( \ytld  - \yt ) } {\sigma} \right\|^2 .
\end{align*}
Next we notice that
\begin{align*}
\frac{\delta}{4n} \left\| A (x^* - \xtldt ) + \frac{n ( \EE[ \ytp ] - \yt ) } {\sigma} \right\|^2  =& \frac{\delta}{4n} \left\| A (x^* - \xt) - \theta A ( \xt - \xtm ) + \frac{ ( \ytld - \yt ) } {\sigma} \right\|^2 \\
\geq& \left( 1 - \frac{1}{\alpha} \right) \frac{\delta}{4n} \left\| A (\xopt - \xt )\right\|^2 \\
&- (\alpha - 1) \frac{\delta}{4n} \left\| \theta A ( \xt - \xtm ) + \frac{ ( \ytld - \yt ) } {\sigma} \right\|^2,
\end{align*}
for some $\alpha>1$ and
\[
\left\| A (x^* - \xt )\right\|^2 \geq \mu^2 \| \xopt - \xt \|^2,
\]
and 
\begin{align*}
\left\| \theta A ( \xt - \xtm ) + \frac{ ( \ytld - \yt ) } {\sigma} \right\|^2 
\geq& -2 \theta^2 \| A ( \xt - \xtm ) \|^2 - \frac{2}{\sigma^2} \| \ytld - \yt \|^2  \\
\geq& - 2 \theta^2 L^2 \| \xt - \xtm \|^2 - \frac{2 n}{\sigma^2} \EE [ \| \ytp - \yt \|^2 ] .
\end{align*}
We follow the same reasoning as in the standard SPDC analysis,
\begin{align*}
& \langle  \ut - u^* + n(\utp - \ut), \xtp - \xtldt \rangle = \frac{ (\ytp - \yopt )^T A (\xtp - \xt) }{n} - \frac{ \theta (\yt - \yopt )^T A (\xt - \xtm) }{n} \\
&+ \frac{(n-1)}{n} (\ytp - \yt)^T A (\xtp - \xt) - \theta (\ytp - \yt)^T A (\xt - \xtm),
\end{align*}
and using Cauchy-Schwartz inequality, we have
\begin{align*}
| (\ytp - \yt)^T A (\xt - \xtm) | \leq& 
\frac{ \| (\ytp - \yt)^T A \|^2 }{ 1/(2\tau) } + 
\frac{ \| \xt - \xtm \|^2 }{ 8 \tau} \\
\leq& \frac{ \| \ytp - \yt \|^2 }{ 1/(2\tau R^2) },
\end{align*}
and 
\begin{align*}
| (\ytp - \yt)^T A (\xtp - \xt) | \leq& 
\frac{ \| (\ytp - \yt)^T A \|^2 }{ 1/(2\tau) } + 
\frac{ \| \xtp - \xt \|^2 }{ 8 \tau} \\
\leq& \frac{ \| \ytp - \yt \|^2 }{ 1/(2\tau R^2) }.
\end{align*}
Thus we get
\begin{align*}
& \langle  \ut - u^* + n(\utp - \ut), \xtp - \xtldt \rangle  \geq 
\frac{ (\ytp - \yopt )^T A (\xtp - \xt) }{n} - \frac{ \theta (\yt - \yopt )^T A (\xt - \xtm) }{n} \\
& - \frac{  \| \ytp - \yt \|^2  }{1/(4\tau R^2)} - \frac{ \| \xtp - \xt \|^2 }{8 \tau}
- \frac{ \theta \| \xt - \xtm \|^2 }{ 8 \tau}  .
\end{align*}
Putting everything together, we have 
\begin{align*}
& \left( \frac{1}{2 \tau} - \frac{ (1 - 1/\alpha) \delta \mu^2}{4n}  \right) \| \xt - \xopt \|^2  +  \left(  \frac{1}{2 \sigma} + \frac{(n-1)\gamma}{4n} \right) \| \yt - \yopt \|^2  +  \theta ( \Lagr(\xt,\yopt) - \Lagr(\xopt,\yopt) ) \\
&+ (n-1) (\Lagr(\xopt,\yopt) - \Lagr(\xopt,\yt) ) + \theta \left( \frac{1}{8 \tau} + \frac{(\alpha - 1) \theta \delta L^2}{2n} \right) \| \xt - x^{t-1} \|^2 + \frac{ \theta (\yt - \yopt )^T A (\xt - \xtm) }{n} \\
&\geq \left( \frac{1}{2\tau} + \frac{\lambda}{2} \right) \EE [ \| \xtp - \xopt \|^2 ] 
+ \left(  \frac{1}{2 \sigma} + \frac{\gamma}{4}  \right) \EE [ \| \ytp - \yopt \|^2 ] + 
\frac{ \EE [ (\ytp - \yopt )^T A (\xtp - \xt) ] }{n}
\\
&+ \EE [ \Lagr(\xtp,\yopt) - \Lagr(\xopt,\yopt) + n( \Lagr(\xopt,\yopt) - \Lagr(\xopt,\ytp) ) ]   \\
& + \left( \frac{1}{2\tau} - \frac{1}{8 \tau} \right) \EE [ \| \xtp - \xt \|^2 ]  
 \\
&+ \left( \frac{1}{2 \sigma} -  4 R^2 \tau - \frac{(\alpha - 1) \delta }{2 \sigma^2} \right) \EE [ \| \ytp - \yt \|^2 ] .
\end{align*}
If we choose the parameters as
\[
\alpha = \frac{\sigma}{4 \delta } + 1, \quad \sigma \tau = \frac{1}{16 R^2},
\]
then we know
\[
\frac{1}{2 \sigma} -  4 R^2 \tau - \frac{(\alpha - 1) \delta }{2 \sigma^2} = \frac{1}{2\sigma} - \frac{1}{4 \sigma} - \frac{1}{ 8 \sigma} > 0,
\]
and 
\[
\frac{(\alpha - 1) \theta \delta L^2}{2n} \leq \frac{\sigma L^2}{8 n^2} \leq \frac{\sigma R^2}{8} \leq \frac{1}{256 \tau},
\]
thus
\[
\frac{1}{8 \tau} + \frac{(\alpha - 1) \theta \delta L^2}{2n} \leq  \frac{3}{8 \tau}.
\]
In addition, we have
\[
1 - \frac{1}{\alpha} = \frac{\sigma }{\sigma + 4 \delta}.
\]
Finally we obtain
\begin{align*}
& \left( \frac{1}{2 \tau} - \frac{\sigma \delta \mu^2 }{ 4 n(\sigma + 4 \delta)}  \right) \| \xt - \xopt \|^2  +  \left(  \frac{1}{2 \sigma} + \frac{(n-1)\gamma}{4n} \right) \| \yt - \yopt \|^2  +  \theta ( \Lagr(\xt,\yopt) - \Lagr(\xopt,\yopt) ) \\
&+ (n-1) (\Lagr(\xopt,\yopt) - \Lagr(\xopt,\yt) ) + \theta \cdot \frac{3}{8 \tau} \| \xt - \xtm \|^2 + \frac{ \theta (\yt - \yopt )^T A (\xt - \xtm) }{n} \\
&\geq \left( \frac{1}{2\tau} + \frac{\lambda}{2} \right) \EE [ \| \xtp - \xopt \|^2 ] 
+ \left(  \frac{1}{2 \sigma} + \frac{\gamma}{4}  \right) \EE [ \| \ytp - \yopt \|^2 ] + 
\frac{ \EE [ (\ytp - \yopt )^T A (\xtp - \xt) ] }{n}
\\
&+ \EE [ \Lagr(\xtp,\yopt) - \Lagr(\xopt,\yopt) + n( \Lagr(\xopt,\yopt) - \Lagr(\xopt,\ytp) ) ]   + \frac{3}{8 \tau} \EE [ \| \xtp - \xt \|^2 ] .
\end{align*}
Now we can define $\theta_x$ and $\theta_y$ as the ratios between the coefficients in the $x$-distance and $y$-distance terms, and let 
$\theta=\max\{\theta_x,\theta_y\}$ as before.
Choosing the step-size parameters as
\begin{align*}
\tau = \frac{1}{4R} \sqrt{ \frac{\gamma}{n \lambda + \delta \mu^2} }, \quad  \sigma = \frac{1}{4R} \sqrt{ \frac{n \lambda + \delta \mu^2}{\gamma  }}
\end{align*}
gives the desired result.

\section{Proof of Theorem \ref{thm:DF-SPDC-convergence}}
\label{sec:proof_adf_spdc}

In this setting, for $i$-th coordinate of the dual variables $y$ we choose $h = \phi_i^*$, let
\[
\Bdivi(y_i,y'_i) = \phi_i^*(y_i) - \phi_i^*(y'_i) + 
\langle  (\phi_i^*)'(y'_i), y_i - y'_i \rangle,
\]
and define
\[
\Bdiv(y,y') = \sum_{i=1}^n \Bdivi(y_i,y'_i).
\]
For $i = 1,\ldots,n$, let $\yitld$ be
\[
    \yitld = \arg\min_{y} \left\{ \phi_i^*(y) + \frac{\Bdivi(y,\yit)}{\sigma} - y \langle a_i, \xtldt \rangle \right\}.
\]
Based on the first-order optimality condition, we have
\[
\langle a_i, \xtldt \rangle - \frac{(\phi_i^*)'(\yitld) - (\phi_i^*)'(\yit)}{\sigma} \in (\phi_i^*)'(\yitld).
\]
Also since $\yiopt$ minimizes $\phi_i^*(y) - y \langle a_i, x^* \rangle$, 
we have
\[
\langle a_i, x^* \rangle \in (\phi_i^*)'(\yiopt).
\]
Using Lemma \ref{lem:sc-min-3point} with $\rho = 1/2$, we obtain
\begin{align*}
- \yiopt \langle a_i, \xtldt \rangle + 
\phi^*_i(\yiopt) + \frac{\Bdivi(\yiopt,\yit)}{\sigma}  \geq& \left( \frac{1}{\sigma} + \frac{1}{2} \right) \Bdivi(\yiopt,\yitld) + \phi^*_i(\yitld) - \yitld \langle a_i, \xtldt \rangle \\
&+\frac{\Bdivi(\yitld, \yit) }{ \sigma} + \frac{\delta}{4} ( (\phi_i^*)'(\yitld) - (\phi_i^*)'(\yiopt) )^2,
\end{align*}
and rearranging terms, we get
\begin{align}
\frac{\Bdivi(\yiopt,\yit)}{\sigma}   \geq& \left( \frac{1}{\sigma} + \frac{1}{2} \right) \Bdivi(\yiopt,\yitld) + \frac{\Bdivi(\yitld, \yit) }{ \sigma}  - (\yitld - \yiopt) \langle a_i, \xtldt \rangle +  ( \phi_i^*(\yitld) -  \phi_i^*(\yiopt) ) \nonumber \\
&+ \frac{\delta}{4} ( (\phi_i^*)'(\yitld) - (\phi_i^*)'(\yiopt) )^2.
\label{eqn:y-main-ineq-BD}
\end{align}
With i.i.d.\ random sampling at each iteration, we have the following relations:
\begin{align*}
\EE [ \yitp ] =& \frac{1}{n} \cdot \yitld + \frac{n-1}{n} \cdot \yit, \\
\EE [ \Bdivi(\yitp,\yiopt) ] =& \frac{ \Bdivi(\yitld,\yiopt) }{n} + \frac{ (n-1) \Bdivi(\yit,\yiopt) }{n}, \\
\EE [ \Bdivi(\yitp,\yit) ] =& \frac{\Bdivi(\yitld,\yit)}{n} , \\
\EE [ \phi_i^*(\yitp) ] =& \frac{1}{n}\cdot  \phi_i^*(\yitld) + \frac{n-1}{n} \cdot \phi_i^*(\yit).
\end{align*}
Plugging the above relations into \eqref{eqn:y-main-ineq-BD} and dividing both sides by $n$, we have
\begin{align*}
\left(  \frac{1}{\sigma} + \frac{(n-1)}{2n} \right) \Bdivi(\yit,\yiopt) \geq& 
\left(  \frac{1}{ \sigma} + \frac{1}{2} \right) \Bdivi(\yitp,\yiopt) + \frac{1}{ \sigma} \EE [ \Bdivi(\yitp,\yit) ] \\
&- \left(  \EE [ ( \yitp - \yit) ] + \frac{1}{n} (\yit - \yiopt)  \right) \langle a_i, \xtldt \rangle \\
&+ \EE [ \phi_i^*(\yitp) ] - \phi_i^*(\yit) + \frac{1}{n} ( \phi_i^*(\yit) - \phi_i^*(\yiopt) ) \\
&+ \frac{\delta}{4n} \left( \langle a_i, \xtldt - \xopt \rangle - \frac{((\phi_i^*)'(\yitld) - (\phi_i^*)'(\yit))}{\sigma}  \right)^2,
\end{align*}
and summing over $i=1,\ldots,n$, we get
\begin{align*}
\left(  \frac{1}{\sigma} + \frac{(n-1)}{2n} \right) \Bdiv(\yt,\yopt) \geq& 
\left(  \frac{1}{ \sigma} + \frac{1}{2}  \right) \EE [ \Bdiv(\ytp,\yopt) ]
+ \frac{ \EE [  \Bdiv(\ytp,\yt) ]}{ \sigma}  \\
&+ \phi_k^*(\yktp) - \phi_k^*(\ykt) + \frac{1}{n} \sum_{i=1}^n ( \phi_i^*(\yit) - \phi_i^*(\yiopt) )\\
&- \left\langle  n(\utp - \ut) + ( \ut - u^*), \xtldt \right\rangle \\
&+ \frac{\delta}{4n} \left\| A (x^* - \xtldt ) + \frac{(  \phi^{*'}(\ytld)  - \phi^{*'}(\yt) ) } {\sigma} \right\|^2,
\end{align*}
where $\phi^{*'}(\yt)$ is a $n$-dimensional vector such that the $i$-th coordinate is 
\[
[\phi^{*'}(\yt)]_i = (\phi_i^*)'(\yit),
\]
and
\[
\ut = \frac{1}{n} \sum_{i=1}^n \yit a_i, \quad \utp = \frac{1}{n} \sum_{i=1}^n \yitp a_i, \quad {\rm and } \quad u^* = \frac{1}{n} \sum_{i=1}^n \yiopt a_i.
\]

On the other hand, since $\xtp$ minimizes a $\frac{1}{\tau} + \lambda$-strongly convex objective
\[
g(x) + \left \langle \ut + n(\utp - \ut), x \right \rangle + \frac{\| x - \xt \|^2}{2 \tau},
\]
we can apply Lemma \ref{lem:sc-min-3point} with $\rho = 0$ to obtain 
\begin{align*}
&g(\xopt) + \langle \ut + n( \utp - \ut), \xopt \rangle + \frac{ \| \xt - \xopt \|^2 }{2 \tau} \\
&\geq g(\xtp) + \langle \ut + n( \utp - \ut), \xtp \rangle + \frac{  \| \xtp - \xt \|^2  }{2 \tau} + \left( \frac{1}{2\tau} + \frac{\lambda}{2} \right) \| \xtp - \xopt \|^2 ,
\end{align*}
and rearranging terms, we get
\begin{align*}
\frac{ \| \xt - \xopt \|^2 }{2 \tau} \geq& \left( \frac{1}{2\tau} + \frac{\lambda}{2} \right) \EE [ \| \xtp - \xopt \|^2 ] + \frac{ \EE [ \| \xtp - \xt \|^2 ] }{2 \tau} 
+ \EE [ g(\xtp) - g(\xopt) ] \\
&+ \EE [ \langle \ut + n( \utp - \ut), \xtp - \xopt \rangle ] .
\end{align*}
Notice that
\begin{align*}
&\Lagr(\xtp,\yopt) - \Lagr(\xopt,\yopt) + n( \Lagr(\xopt,\yopt) - \Lagr(\xopt,\ytp) ) - (n-1) (\Lagr(\xopt,\yopt) - \Lagr(\xopt,\yt) ) \\
&= \frac{1}{n} \sum_{i=1}^n ( \phi_i^*(\yit) - \phi_i^*(\yopt) )
+ (\phi_k^*(\yktp) - \phi_k^*(\ykt) ) + g(\xtp) - g(\xopt) \\
&+ \langle \uopt, \xtp \rangle - \langle \ut, \xopt \rangle + n \langle \ut - \utp, \xopt \rangle,
\end{align*}
so
\begin{align*}
&\frac{ \| \xt - \xopt \|^2 }{2 \tau} + \left(  \frac{1}{\sigma} + \frac{(n-1)}{2n} \right) \Bdiv(\yt,\yopt)  + (n-1) (\Lagr(\xopt,\yopt) - \Lagr(\xopt,\yt) ) \\
&\geq \left( \frac{1}{2\tau} + \frac{\lambda}{2} \right) \EE [ \| \xtp - \xopt \|^2 ]
+ \left(  \frac{1}{ \sigma} + \frac{1}{2}  \right) \EE [ \Bdiv(\ytp,\yopt) ] + 
\frac{ \EE [ \| \xtp - \xt \|^2 ] }{2 \tau}  + \frac{ \EE [  \Bdiv(\ytp,\yt) ]}{ \sigma}\\
&+ \EE [ \Lagr(\xtp,\yopt) - \Lagr(\xopt,\yopt) + n( \Lagr(\xopt,\yopt) - \Lagr(\xopt,\ytp) ) ]   \\
&+ \EE [ \langle  \ut - u^* + n( \utp - \ut), \xtp - \bar x^t \rangle] + \frac{\delta}{4n} \left\| A (x^* - \xtldt ) + \frac{ (  \phi^{*'}(\ytld)  - \phi^{*'}(\yt) ) } {\sigma} \right\|^2.
\end{align*}
Next, we have
\begin{align*}
\frac{\delta}{4n} \left\| A (x^* - \xtldt ) + \frac{(  \phi^{*'}(\ytld)  - \phi^{*'}(\yt) ) } {\sigma} \right\|^2  =& \frac{\delta}{4n} \left\| A (x^* - \xt) - \theta A ( \xt - \xtm ) + \frac{ (  \phi^{*'}(\ytld)  - \phi^{*'}(\yt) ) } {\sigma} \right\|^2 \\
\geq& \left( 1 - \frac{1}{\alpha} \right) \frac{\delta}{4n} \left\| A (\xopt - \xt )\right\|^2 \\
&- (\alpha - 1) \frac{\delta}{4n} \left\| \theta A ( \xt - \xtm ) + \frac{ ( \phi^{*'}(\ytld)  - \phi^{*'}(\yt) ) } {\sigma} \right\|^2,
\end{align*}
for any $\alpha>1$ and
\[
\left\| A (x^* - \xt )\right\|^2 \geq \mu^2 \| \xopt - \xt \|^2,
\]
and 
\begin{align*}
\left\| \theta A ( \xt - \xtm ) + \frac{ ( \phi^{*'}(\ytld)  - \phi^{*'}(\yt) ) } {\sigma} \right\|^2 
\geq& -2 \theta^2 \| A ( \xt - \xtm ) \|^2 - \frac{2}{\sigma^2}  \| \phi^{*'}(\ytld) - \phi^{*'}(\yt) \|^2 ] \\
\geq& - 2 \theta^2 L^2 \| \xt - \xtm \|^2 - \frac{2 n}{\sigma^2} \EE [ \| \phi^{*'}(\ytp) - \phi^{*'}(\yt) \|^2 ].
\end{align*}
Following the same reasoning as in the standard SPDC analysis, we have
\begin{align*}
& \langle  \ut - u^* + n(\utp - \ut), \xtp - \xtldt \rangle = \frac{ (\ytp - \yopt )^T A (\xtp - \xt) }{n} - \frac{ \theta (\yt - \yopt )^T A (\xt - \xtm) }{n} \\
&+ \frac{(n-1)}{n} (\ytp - \yt)^T A (\xtp - \xt) - \theta (\ytp - \yt)^T A (\xt - \xtm),
\end{align*}
and using Cauchy-Schwartz inequality, we have
\begin{align*}
| (\ytp - \yt)^T A (\xt - \xtm) | \leq& 
\frac{ \| (\ytp - \yt)^T A \|^2 }{ 1/(2\tau) } + 
\frac{ \| \xt - \xtm \|^2 }{ 8 \tau} \\
\leq& \frac{ \| \ytp - \yt \|^2 }{ 1/(2\tau R^2) },
\end{align*}
and 
\begin{align*}
| (\ytp - \yt)^T A (\xtp - \xt) | \leq& 
\frac{ \| (\ytp - \yt)^T A \|^2 }{ 1/(2\tau) } + 
\frac{ \| \xtp - \xt \|^2 }{ 8 \tau} \\
\leq& \frac{ \| \ytp - \yt \|^2 }{ 1/(2\tau R^2) }.
\end{align*}
Thus we get
\begin{align*}
& \langle  \ut - u^* + n(\utp - \ut), \xtp - \xtldt \rangle  \geq 
\frac{ (\ytp - \yopt )^T A (\xtp - \xt) }{n} - \frac{ \theta (\yt - \yopt )^T A (\xt - \xtm) }{n} \\
& - \frac{  \| \ytp - \yt \|^2  }{1/(4\tau R^2)} - \frac{ \| \xtp - \xt \|^2 }{8 \tau}
- \frac{ \theta \| \xt - \xtm \|^2 }{ 8 \tau}.
\end{align*}
Also we can lower bound the term $\Bdiv(\ytp,\yt)$ using Lemma \ref{lem:sc-min-3point} with $\rho = 1/2$:
\begin{align*}
\Bdiv(\ytp,\yt) =& \sum_{i=1}^n \left( \phi_i^*(\yitp) - \phi_i^*(\yit) - \langle (\phi_i^*)'(\yit) , \yitp - \yit \rangle \right) \\
\geq& \sum_{i=1}^n \left( \frac{\gamma}{2} (\yitp - \yit)^2 + \frac{\delta}{2} ( (\phi_i^*)'(\yitp) - (\phi_i^*)'(\yit) )^2 \right) \\
=& \frac{\gamma}{2} \| \ytp - \yt \|^2 + \frac{\delta}{2} \| \phi^{*'}(\ytp) - \phi^{*'}(\yt) \|^2.
\end{align*}
Combining everything above together, we have 
\begin{align*}
& \left( \frac{1}{2 \tau} - \frac{ (1 - 1/\alpha) \delta \mu^2}{4n}  \right) \| \xt - \xopt \|^2  +  \left(  \frac{1}{ \sigma} + \frac{(n-1)}{2n} \right) \Bdiv(\yt,\yopt)  +  \theta ( \Lagr(\xt,\yopt) - \Lagr(\xopt,\yopt) ) \\
&+ (n-1) (\Lagr(\xopt,\yopt) - \Lagr(\xopt,\yt) ) + \theta \left( \frac{1}{8 \tau} + \frac{(\alpha - 1) \theta \delta L^2}{2n} \right) \| \xt - x^{t-1} \|^2 + \frac{ \theta (\yt - \yopt )^T A (\xt - \xtm) }{n} \\
&\geq \left( \frac{1}{2\tau} + \frac{\lambda}{2} \right) \EE [ \| \xtp - \xopt \|^2 ] 
+ \left(  \frac{1}{ \sigma} + \frac{1}{2}  \right) \EE [ \Bdiv(\ytp,\yopt) ] + 
\frac{ \EE [ (\ytp - \yopt )^T A (\xtp - \xt) ] }{n}
\\
&+ \EE [ \Lagr(\xtp,\yopt) - \Lagr(\xopt,\yopt) + n( \Lagr(\xopt,\yopt) - \Lagr(\xopt,\ytp) ) ]   \\
& + \left( \frac{1}{2\tau} - \frac{1}{8 \tau} \right) \EE [ \| \xtp - \xt \|^2 ]  + \left( \frac{\gamma}{2 \sigma} -  4 R^2 \tau  \right) \EE [ \| \ytp - \yt \|^2 ] \\
& + \left( \frac{\delta}{2 \sigma} - \frac{(\alpha - 1) \delta }{2 \sigma^2} \right)  \EE [ \| \phi^{*'}(\ytp) - \phi^{*'}(\yt) \|^2 ] .
\end{align*}
If we choose the parameters as
\[
\alpha = \frac{\sigma}{4 } + 1, \quad \sigma \tau = \frac{\gamma}{16 R^2},
\]
then we know
\[
\frac{\gamma}{2 \sigma} -  4 R^2 \tau = \frac{\gamma}{2 \sigma} - \frac{\gamma}{4 \sigma}  > 0,
\]
and
\[
\frac{\delta}{2 \sigma} - \frac{(\alpha - 1) \delta }{2 \sigma^2} = \frac{\delta}{2 \sigma} - \frac{\delta}{8 \sigma} > 0
\]
and 
\[
\frac{(\alpha - 1) \theta \delta L^2}{2n} \leq \frac{\sigma \delta L^2}{8 n^2} \leq \frac{\delta \sigma R^2}{8} \leq \frac{\delta \gamma}{256 \tau} \leq \frac{1}{256 \tau},
\]
thus
\[
\frac{1}{8 \tau} + \frac{(\alpha - 1) \theta \delta L^2}{2n} \leq  \frac{3}{8 \tau}.
\]
In addition, we have
\[
1 - \frac{1}{\alpha} = \frac{\sigma }{\sigma + 4}.
\]
Finally we obtain
\begin{align*}
& \left( \frac{1}{2 \tau} - \frac{\sigma \delta \mu^2 }{ 4 n(\sigma + 4)}  \right) \| \xt - \xopt \|^2  +  \left(  \frac{1}{\sigma} + \frac{(n-1)}{2n} \right) \Bdiv(\yt,\yopt)  +  \theta ( \Lagr(\xt,\yopt) - \Lagr(\xopt,\yopt) ) \\
&+ (n-1) (\Lagr(\xopt,\yopt) - \Lagr(\xopt,\yt) ) + \theta \cdot \frac{3}{8 \tau} \| \xt - \xtm \|^2 + \frac{ \theta (\yt - \yopt )^T A (\xt - \xtm) }{n} \\
&\geq \left( \frac{1}{2\tau} + \frac{\lambda}{2} \right) \EE [ \| \xtp - \xopt \|^2 ] 
+ \left(  \frac{1}{\sigma} + \frac{1}{2}  \right) \EE [ \| \ytp - \yopt \|^2 ] + 
\frac{ \EE [ (\ytp - \yopt )^T A (\xtp - \xt) ] }{n}
\\
&+ \EE [ \Lagr(\xtp,\yopt) - \Lagr(\xopt,\yopt) + n( \Lagr(\xopt,\yopt) - \Lagr(\xopt,\ytp) ) ]   + \frac{3}{8 \tau} \EE [ \| \xtp - \xt \|^2 ] .
\end{align*}
As before, we can define $\theta_x$ and $\theta_y$ as the ratios between the coefficients in the $x$-distance and $y$-distance terms, and let
$\theta=\max\{\theta_x,\theta_y\}$.
Then choosing the step-size parameters as
\begin{align*}
\tau = \frac{1}{4R} \sqrt{ \frac{\gamma}{n \lambda + \delta \mu^2} }, \quad  \sigma = \frac{1}{4R} \sqrt{ \gamma(n \lambda + \delta \mu^2)}
\end{align*}
gives the desired result.

\end{document}